\documentclass[a4paper]{amsart}

\usepackage{amssymb} 
\usepackage[all]{xy}
\usepackage{dsfont}
\usepackage{hyperref}

\SelectTips{cm}{}
\calclayout 

\thanks{Version from 2nd February 2011.}

\numberwithin{equation}{section}

\theoremstyle{plain}
\newtheorem{theorem}[equation]{Theorem}
\newtheorem{proposition}[equation]{Proposition}
\newtheorem{lemma}[equation]{Lemma} 
\newtheorem{corollary}[equation]{Corollary}

\theoremstyle{definition}

\newtheorem{example}[equation]{Example}

\theoremstyle{remark} 
\newtheorem{remark}[equation]{Remark} 

\newcommand{\cl}{\operatorname{cl}}

\newcommand{\Coloc}{\operatorname{Coloc}}
\newcommand{\Cone}{\operatorname{Cone}}
\newcommand{\cosupp}{\operatorname{cosupp}}

\newcommand{\End}{\operatorname{End}}
\newcommand{\Ext}{\operatorname{Ext}}

\newcommand{\fHom}{\operatorname{\mathcal{H}\!\!\;\mathit{om}}}

\newcommand{\Hom}{\operatorname{Hom}}
\newcommand{\id}{\operatorname{id}}
\newcommand{\Id}{\operatorname{Id}}
\renewcommand{\Im}{\operatorname{Im}}
\newcommand{\Inj}{\operatorname{\mathsf{Inj}}}
\newcommand{\Ker}{\operatorname{Ker}}
\newcommand{\KInj}[1]{\mathsf K(\Inj #1)}

\newcommand{\Loc}{\operatorname{Loc}}

\newcommand{\Mod}{\operatorname{\mathsf{Mod}}}

\newcommand{\res}{\operatorname{res}}
\newcommand{\RHom}{\operatorname{{\mathbf R}Hom}}

\newcommand{\Spec}{\operatorname{Spec}}
\newcommand{\StMod}{\operatorname{\mathsf{StMod}}}

\newcommand{\supp}{\operatorname{supp}}

\newcommand{\Thick}{\operatorname{Thick}}

\newcommand{\inc}{\mathrm{inc}} 
\newcommand{\op}{\mathrm{op}}

\newcommand{\col}{\colon}
\newcommand{\da}{{\downarrow}}
\newcommand{\ges}{{\scriptscriptstyle\geqslant}}
\newcommand{\hh}[1]{H^{*}#1} 

\newcommand{\kos}[2]{{#1}/\!\!/{#2}} 

\newcommand{\lotimes}{\otimes^{\mathbf L}}
\newcommand{\Lra}{\Leftrightarrow}
 
\newcommand{\lto}{\longrightarrow}
\renewcommand{\setminus}{\smallsetminus}
\newcommand{\ua}{{\uparrow}}
\newcommand{\wt}{\widetilde}
\newcommand{\xla}{\xleftarrow}
\newcommand{\xra}{\xrightarrow}

\newcommand{\bik}{Benson/Iyengar/Krause}

\def\sfT{\mathcal{T}}
\def\mcU{\mathcal{U}} 
\def\mcV{\mathcal{V}}
\def\mcW{\mathcal{W}} 
 
\def\mcZ{\mathcal{Z}}

\def\sfc{\mathsf c}

\def\sfi{\mathsf i}
 
\def\sfq{\mathsf q}

\def\sfC{\mathsf C}
\def\sfD{\mathsf D} 
 
\def\sfG{\mathsf G}
\def\sfK{\mathsf K}

\def\sfS{\mathsf S} 
\def\sfT{\mathsf T} 
\def\sfU{\mathsf U}

\def\bbQ{\mathbb Q} 
\def\bbZ{\mathbb Z}

\newcommand{\bsz}{\boldsymbol{z}}

\newcommand{\fa}{\mathfrak{a}} 

\newcommand{\fm}{\mathfrak{m}} 
\newcommand{\fp}{\mathfrak{p}}
\newcommand{\fq}{\mathfrak{q}}

\newcommand{\gam}{\varGamma} 
\newcommand{\lam}{\varLambda}

\def\Si{\Sigma} 
\def\si{\sigma}

\def\one{\mathds 1}

\newcommand{\bloc}{{L}} 

\newcommand{\cent}{{Z}} 

\title[Colocalizing subcategories and cosupport]
{Colocalizing subcategories and cosupport}

\author{Dave Benson} 
\address{Dave Benson \\ 
Institute of Mathematics\\ 
University of Aberdeen\\ 
King's College\\ 
Aberdeen AB24 3UE\\ 
Scotland U.K.}

\author{Srikanth B. Iyengar} 
\address{Srikanth B. Iyengar\\ 
Department of Mathematics\\ 
University of Nebraska\\ 
Lincoln, NE 68588\\ 
U.S.A.}

\author{Henning Krause} 
\address{Henning Krause\\ 
Fakult\"at f\"ur Mathematik\\ 
Universit\"at Bielefeld\\ 
33501 Bielefeld\\ 
Germany.}

\begin{document}

\begin{abstract}
The Hom closed colocalizing subcategories of the stable module
category of a finite group are classified. Along the way, the
colocalizing subcategories of the homotopy category of injectives over
an exterior algebra, and the derived category of a formal commutative
differential graded algebra, are classified. To this end, and with an eye
towards future applications, a notion of local homology and cosupport
for triangulated categories is developed, building on earlier work of the authors on
local cohomology and support.
\end{abstract}

\keywords{colocalizing subcategory, cosupport, local homology,
  localizing subcategory, stable module category, triangulated
  category}

\subjclass[2010]{20J06(primary); 13D45, 16E45, 18E30}

\thanks{The research of the first and second authors was undertaken
  during visits to the University of Paderborn, each supported by a
  research prize from the Humboldt Foundation. The research of the
  second author was also partly supported by NSF grant DMS 0903493.}

\maketitle \setcounter{tocdepth}{1} \tableofcontents

\section{Introduction}
Let $G$ be a finite group and $k$ a field of characteristic $p$,
dividing the order of $G$, and $\StMod(kG)$ the stable module category of
possibly infinite dimensional $kG$-modules. We write $\mcV_{G}$ for
the set of all homogeneous prime ideals except the maximal ideal in
the cohomology algebra $H^*(G,k)$ of $G$, and $\mcV_{G}(M)$ for the
support of any $M\in\StMod(kG)$, defined by Benson, Carlson and
Rickard \cite{Benson/Carlson/Rickard:1996a} when $k$ is algebraically
closed, and extended in \cite{\bik:2008b} to all fields. One of the
main results in this work is a classification of the colocalizing
subcategories of $\StMod(kG)$:

\begin{theorem}
\label{ith:costratification}
The map that assigns to each subset $\mcU\subseteq\mcV_{G}$ the
subcategory
\[
\{N\in\StMod(kG)\mid {\underline\Hom}_{kG}(M,N)=0\text{ for all $M$
  with $\mcV_{G}(M)\subseteq \mcU$}\}
\]
gives a bijection between subsets of $\mcV_{G}$ and colocalizing
subcategories of $\StMod(kG)$ that are closed under tensor product
with simple $kG$-modules.
\end{theorem}

A \emph{colocalizing subcategory} $\sfC$ is by definition a full
triangulated subcategory that is closed under set-indexed
products. Such a subcategory is closed under tensor product with
simples if and only if it is Hom closed: If $N$ is in $\sfC$, so is
$\Hom_{k}(M,N)$ for any
$M\in\StMod(kG)$. Theorem~\ref{ith:costratification} complements the
classification of the localizing subcategories of $\StMod(kG)$ from
\cite[Theorem~10.3]{\bik:2008b}. Combining them gives a remarkable
bijection:

\begin{corollary}
\label{ico:locandcoloc}
The map sending a localizing subcategory $\sfS$ of $\StMod(kG)$ to
$\sfS^\perp$ induces a bijection
\[ 
\left\{\begin{gathered} \text{tensor closed
  localizing}\\ \text{subcategories of $\StMod(kG)$}
\end{gathered}\;
\right\} \xymatrix@C=3pc{ \ar[r]^-{\sim} &} \left\{
\begin{gathered}
  \text{Hom closed colocalizing}\\ \text{subcategories of
    $\StMod(kG)$}
\end{gathered}
\right\}\,.
\]
The inverse map sends a colocalizing subcategory $\sfS$ to
$^\perp\sfS$. \qedhere
\end{corollary}

Theorem~\ref{ith:costratification} and
Corollary~\ref{ico:locandcoloc}, proved in
Section~\ref{se:finitegroups}, are analogues of recent results of
Neeman~\cite{Neeman:coloc} on the derived category of a noetherian
commutative ring.

The definition of the inverse of the map in
Theorem~\ref{ith:costratification} involves a notion of
\emph{cosupport} for a module $M$ in $\StMod(kG)$, introduced in this
work to be the subset
\[ 
\cosupp_GM=\{\fp\in\mcV_{G}\mid \Hom_k(\kappa_\fp,M)\text{ is not
  projective}\},
\]
with $\kappa_{\fp}$ the Rickard idempotent module associated to $\fp$,
constructed in \cite{Benson/Carlson/Rickard:1996a}. Recall that the
support of $M$ is $\{\fp\in\mcV_{G}\mid \kappa_\fp\otimes M\text{ is
  not projective}\}$. The inverse map in
Theorem~\ref{ith:costratification} assigns to a subcategory $\sfC$ of
$\StMod(kG)$ the complement in $\mcV_{G}$ of the set
$\bigcup_{M\in\sfC}\cosupp_{G}M$.

The proof of Theorem~\ref{ith:costratification} is modelled on that of
\cite[Theorem~10.3]{\bik:2008b}, where localizing subcategories of
$\StMod(kG)$ are classified. It involves a sequence of changes of
category, for which reason it has been necessary to develop a theory
of cosupport for objects in triangulated categories, along the lines
for the one for support in our earlier work
\cite{\bik:2008a,\bik:2008b,\bik:2009a}.

The context is a compactly generated triangulated category $\sfT$ with
set-indexed coproducts endowed with an action of a graded commutative
noetherian ring $R$; meaning, a homomorphism $R\to \cent^{*}(\sfT)$ of
graded rings from $R$ to the graded center of $\sfT$. For each $\fp$
in $\Spec R$, the set of homogeneous prime ideals in $R$, we introduce
a \emph{local homology functor} $\lam^\fp$, constructed as a right
adjoint to the local cohomology functor, $\gam_{\fp}$, from
\cite{\bik:2008a}, and define the \emph{cosupport} of an object $X$ of
$\sfT$ by
\[
\cosupp_{R}X = \{\fp\in\Spec R\mid\lam^\fp X\ne 0\}\,.
\]
In the first part of this paper, Sections \ref{se:localizations} to
\ref{se:plocalpcomplete}, we establish salient properties of local
homology and cosupport; for instance, that the maximal elements with
respect to inclusion in the cosupport and the support of any object
$X$ in $\sfT$ coincide:
\[
\max(\cosupp_{R}X) = \max(\supp_{R}X)\,.
\]
This is proved as part of Theorem~\ref{th:max}. It follows that
$\cosupp_{R}X=\varnothing$ if and only if $\supp_{R}X=\varnothing$,
which is equivalent to $X=0$ by \cite[Theorem~5.2]{\bik:2008a}. These
results suggest a close connection between the support and
cosupport. However, while the support of an object is well-understood,
the cosupport remains a mysterious entity. For instance, the only
complete results we could obtain for finitely generated modules over
commutative noetherian rings are given in
Propositions~\ref{pr:cosuppZ} and \ref{pr:cosuppA}.

From Section~\ref{se:tens} onwards we turn to colocalizing
subcategories of $\sfT$, focusing on the case when $\sfT$ is tensor
triangulated with a \emph{canonical $R$-action}, meaning an action
induced by a homomorphism $R\to \End^{*}_{\sfT}(\one)$, where $\one$
is the unit for the tensor product on $\sfT$. This is the context of
the main results of this work, and the rest of this introduction. The
category $\sfT$ admits an internal function object, denoted
$\fHom(X,Y)$, and it is natural to examine the Hom closed colocalizing
subcategories of $\sfT$. A useful result concerning these is that for
each $X\in\sfT$ there is an equality
\[
\Coloc^{\fHom}_{\sfT}(X) = \Coloc^{\fHom}_{\sfT}(\{\lam^{\fp}X\mid
\fp\in\Spec R\})
\]
which is a form of local-global principle for colocalizing
subcategories. This statement is Theorem~\ref{th:tensor-locglob} and
an analogue of such a local-global principle for localizing
subcategories in \cite[Theorem~3.6]{\bik:2008b}. The theorem is a
first step in our approach to the problem of classifying the Hom
closed colocalizing subcategories of $\sfT$, for it permits one to
reduce it to the classification problem for $\lam^{\fp}\sfT$, the
essential image of the functor $\lam^{\fp}$, for each $\fp\in\Spec R$; see
Proposition~\ref{pr:lg=reduction}. We note that $\lam^{\fp}\sfT$ is
itself colocalizing and Hom closed; see Propositions~\ref{pr:product}
and \ref{pr:natisos}.

The following definition thus naturally emerges: \emph{the tensor
  triangulated category $\sfT$ is costratified by $R$} if for each
$\fp\in\Spec R$ there are no non-trivial Hom closed colocalizing
subcategories in $\lam^\fp\sfT$. Given the discussion above, it is
clear that when this property holds the map assigning to a subcategory
$\sfC$ the subset $\bigcup_{X\in\sfC}\cosupp_{R}X$ of $\Spec R$ sets
up a bijection
\[ 
\left\{
\begin{gathered}
  \text{Hom closed colocalizing}\\ \text{subcategories of $\sfT$}
\end{gathered}\;
\right\} \stackrel{\sim}\lto \left\{
\begin{gathered}
  \text{subsets of $\supp_R\sfT$}
\end{gathered}\;
\right\}.
\] 
This bijection is Corollary~\ref{co:coloc-classify} and was the main
reason for our interest in the costratification condition. However,
there are other remarkable consequences that follow from it.  For
instance, we prove in Theorem~\ref{th:strat-costrat}: \emph{if $\sfT$
  is costratified by $R$, it is also stratified by $R$}, meaning that
there are no proper tensor closed localizing subcategories of
$\gam_{\fp}\sfT$; see \cite{\bik:2008b,\bik:2009a}. One consequence is
that if $\sfT$ is costratified by $R$ then there is a bijection,
analogous to the one in Corollary~\ref{ico:locandcoloc}, between the
tensor closed localizing subcategories and the Hom closed colocalizing
subcategories of $\sfT$, via left and right perp; see
Corollary~\ref{co:loccoloc}.

In Theorem~\ref{th:cosupp-hom} we prove that if $\sfT$ is stratified
by $R$ there is an equality
\[ 
\cosupp_R\fHom(X,Y)=\supp_R X\cap \cosupp_R Y\,\quad\text{for all $X,Y\in\sfT$.}
\]
This is an analogue of the tensor product theorem for support
\cite[Theorem~7.3]{\bik:2009a}.  It follows that one gets
\[
\Hom_\sfT^*(X,Y)=0\quad\iff\quad\supp_RX\cap\cosupp_RY=\varnothing
\]
provided that the tensor identity generates $\sfT$; see
Corollary~\ref{co:hom-vanishing}.  This is a surprisingly complete
result, for it is often difficult to obtain precise conditions under
which there are non-zero maps between objects in a triangulated
category.

In Section~\ref{se:finitegroups} we prove
Theorem~\ref{ith:costratification}, by establishing that the tensor
triangulated category $\StMod(kG)$ is costratified by the canonical
action of $\hh(G,k)$. Along the way we prove that the following tensor
triangulated categories are costratified:
\begin{itemize}
\item The derived category of a formal dg algebra whose cohomology is
  graded commutative and noetherian; see Theorem~\ref{th:cdga}.
\item The homotopy category of graded injectives over an exterior
  algebra, viewed as a dg algebra with zero differential; see
  Theorem~\ref{th:ext-costrat}.
\item 
The homotopy category of complexes of injective $kG$-modules, where
$G$ is a finite group; see Theorem~\ref{th:finite-groups}.
\end{itemize}
The proofs of these results use much of the material on cosupport and
local homology in the preceding sections, as well as results on their
behavior under changes of categories, studied in
Section~\ref{se:change of categories}.  Specialized to the case of a
commutative noetherian ring, viewed as a dg algebra concentrated in
degree $0$, the first item in the preceding list is Neeman's theorem,
mentioned at the beginning, which was the inspiration for the results
described in this article.

\subsection*{Acknowledgments}
It is a pleasure to thank Amnon Neeman for carefully reading a preliminary version of this manuscript, and the referee for suggestions regarding the presentation of this material.

\section{(Co)localization functors on triangulated categories}
\label{se:localizations}
In this section we collect basic facts about localization and
colocalization functors on triangulated categories required in this
work; see \cite[\S3]{Benson/Iyengar/Krause:2008a} for details.

Let $\sfT$ be a triangulated category which admits set-indexed
products and coproducts. We write $\Si$ for the suspension functor on
$\sfT$. The \emph{kernel} of an exact functor $F\col\sfT\to\sfT$ is
the full subcategory
\[
\Ker L=\{X\in\sfT\mid LX=0\}\,,
\]
while the \emph{essential image} of $F$ is the full subcategory
\[
\Im F=\{X\in\sfT\mid X\cong FY\text{ for some $Y$ in $\sfT$}\}.
\]

A \emph{localizing subcategory} of $\sfT$ is a full triangulated
subcategory that is closed under taking all coproducts.  We write
$\Loc_\sfT(\sfC)$ for the smallest localizing subcategory containing a
given class of objects $\sfC$ in $\sfT$, and call it the localizing
subcategory \emph{generated} by $\sfC$. Analogously, a
\emph{colocalizing subcategory} of $\sfT$ is a full triangulated
subcategory that is closed under taking all products, and
$\Coloc_\sfT(\sfC)$ denotes the colocalizing subcategory of $\sfT$
that is \emph{cogenerated} by $\sfC$.

A \emph{localization functor} $L\col\sfT\to\sfT$ is an exact functor
that admits for each $X$ in $\sfT$ a natural morphism $\eta X\col X\to
LX$, called \emph{adjunction}, such that $L(\eta X)$ is an isomorphism
and $L(\eta X)=\eta (LX)$.  A functor $\gam\col\sfT\to\sfT$ is a
\emph{colocalization functor} if its opposite functor
$\gam^\op\col\sfT^\op\to\sfT^\op$ is a localization functor; the
corresponding natural morphism $\theta X\col \gam X\to X$ is called
\emph{coadjunction}.

A localization functor $L\col\sfT\to\sfT$ is essentially determined by
its kernel, which is a localizing subcategory of $\sfT$, for it
coincides with the kernel of a functor that admits a right adjoint; see
\cite[Lemma~3.1]{\bik:2008a}. The natural transformation
$\eta\col\Id_\sfT\to L$ induces for each object $X$ in $\sfT$ a
natural exact \emph{localization triangle}
\begin{equation*}
\gam X\lto X\lto \bloc X\lto
\end{equation*}
This exact triangle gives rise to an exact functor
$\gam\col\sfT\to\sfT$ with
\[
\Ker L = \Im\gam \quad\text{and}\quad \Ker\gam =\Im L.
\]
The functor $\gam$ is a colocalization, and each colocalization
functor on $\sfT$ arises in this way. This yields a natural bijection
between localization and colocalization functors on $\sfT$. Note that
$\Ker\gam$ is a colocalizing subcategory of $\sfT$.

Given a subcategory $\sfC$ of a triangulated category $\sfT$ we define
full subcategories
\begin{align*}
{^\perp}\sfC&=\{X\in\sfT\mid \text{$\Hom_\sfT(X,\Si^{n}Y)=0$ for all
  $Y\in \sfC$ and $n\in\bbZ$.}\} \\ \sfC^{\perp}&=\{X\in\sfT\mid
\text{$\Hom_\sfT^*(\Si^{n}Y,X)=0$ for all $Y\in\sfC$ and
  $n\in\bbZ$}\}.
\end{align*}
Evidently, ${^\perp}\sfC$ is a localizing subcategory, and
$\sfC^\perp$ is a colocalizing subcategory.

The next lemma summarizes the basic facts about localization and
colocalization.
\begin{lemma}\label{le:loc-basic}
Let $\sfT$ be a triangulated category and $\sfS$ a triangulated
subcategory. Then the following are equivalent:
\begin{enumerate}
\item There exists a localization functor $L\col\sfT\to\sfT$ such that
  $\Ker L=\sfS$.
\item There exists a colocalization functor $\gam\col\sfT\to\sfT$ such
  that $\Im\gam=\sfS$.
\end{enumerate}
In that case both functors are related by a functorial exact triangle
\[
\gam X\lto X\lto \bloc X\lto.
\] 
Moreover, there are equalities $\sfS^\perp=\Im L=\Ker\gam$ and
$^\perp(\sfS^\perp)=\sfS$.
\end{lemma}
\begin{proof}
See \cite[Lemma~3.3]{\bik:2008a}.
\end{proof}

\begin{remark}
There is a dual version of Lemma~\ref{le:loc-basic} whose formulation
is left to the reader. Note that $\Im F^\op=\Im F$ and $\Ker
F^\op=\Ker F$ for any functor $F$.
\end{remark}

\subsection*{Adjoints}
We discuss the formal properties of right adjoints of (co)localization
functors. This material is the foundation for local homology and
cosupport.

\begin{proposition}
\label{pr:coloc-fun}
Let $L,\gam\col \sfT\to\sfT$ be exact functors such that $L$ is a
localization functor, $\gam$ is a colocalization functor, and both
induce a functorial exact triangle $\gam X\to X\to LX\to $. Then $L$
admits a right adjoint if and only if $\gam$ admits a right adjoint.
In that case let $\lam$ and $V$ denote right adjoints of $\gam$ and
$L$, respectively.\footnote{It is customary to write $L$ for a
  localization functor. A colocalization functor is a localization
  functor for the opposite category; we denote it $\gam$, thought of
  as $L$ turned upside down. The interpretation of local cohomology in
  the sense of Grothendieck as colocalization provides another reason
  for the use of $\gam$. Local homology in the sense of Greenlees and
  May is denoted $\lam$; it is a right adjoint of $\gam$ and hence a
  localization. The corresponding colocalization is thus denoted
  $V$.} Then the following holds.
\begin{enumerate}
\item The functor $\lam$ is a localization functor and $V$ is a
  colocalization functor.  They induce a functorial exact triangle
\begin{equation*}
V X\lto X\lto \lam X\lto.
\end{equation*}
\item There are identities
\[
(\Im\gam)^\perp=\Ker\gam=\Im L=\Im V=\Ker\lam={^\perp(\Im\lam)}\,.
\]
\item 
There are isomorphisms
\[
\lam\gam\xra{\sim}\lam,\quad\gam\xra{\sim}\gam\lam,\quad
VL\xra{\sim}L,\quad\text{and}\quad V\xra{\sim}LV\,.
\]
\item 
The functors $\gam$ and $\lam$ induce mutually quasi-inverse
equivalences
\[
\Im\lam\xra{\sim}\Im\gam\quad\text{and}\quad\Im\gam\xra{\sim}\Im\lam\,.
\]
\end{enumerate}
\end{proposition}

\begin{remark}
The functors $L,\gam,\lam,V$ occurring in the preceding proposition
induce the following recollement
\[
\xymatrix{\sfS\,\ar[rr]|-\inc&&\,\sfT\,
  \ar[rr]|-Q\ar@<1.35ex>[ll]^-{V}\ar@<-1.35ex>[ll]_-{L}&&
  \,\sfT/\sfS\ar@<1.35ex>[ll]^-{\bar{
      \lam\ }}\ar@<-1.35ex>[ll]_-{\bar{ \gam }}}
\] 
where $\sfS=\Im L=\Im V$ and $Q\col\sfT\to\sfT/\sfS$ denotes the
quotient functor so that $\gam=\bar\gam Q$ and $\lam=\bar\lam Q$.
\end{remark}

\begin{proof}[Proof of Proposition~\ref{pr:coloc-fun}]
It follows from Proposition~\ref{pr:rightadjoint} that $L$ admits a
right adjoint if and only if there exists a colocalization functor
$V\col\sfT\to\sfT$ with $\Im V=\Im L$. Using Lemma~\ref{le:loc-basic}
and the fact that $\Im L=\Ker\gam$, it follows that the existence of
$V$ is equivalent to the existence of a localization functor
$\lam\col\sfT\to\sfT$ with
$\Ker\lam=\Ker\gam$. Proposition~\ref{pr:leftadjoint} and
Remark~\ref{re:leftadjoint} imply that the existence of $\lam$ is
equivalent to the existence of a right adjoint of $\gam$.

(1) The properties of $\lam$ and $V$ are explained above. The
existence of the functorial exact triangle then follows from
Lemma~\ref{le:loc-basic} as $\Ker\lam=\Im V$.

(2) The identities follow from the first part of the proof and
Lemma~\ref{le:loc-basic}.

(3) Combine the localization triangles for $L$ and $\lam$ with the
identities in (2).

(4) The isomorphisms in (3) induce isomorphisms
\[\lam\gam\lam\cong\lam^2\cong\lam\quad\text{and}\quad
\gam\lam\gam\cong\gam^2\cong\gam.\]
Thus $\lam\gam$ is isomorphic to the
identity on $\Im\lam$, while
$\gam\lam$ is isomorphic to the identity on $\Im\gam$.
\end{proof}

\section{Local cohomology and support}
\label{se:support}
In this section we recall the construction, and basic properties, of
local cohomology functors and support for triangulated categories,
from \cite{Benson/Iyengar/Krause:2008a,Benson/Iyengar/Krause:2009a}.

\subsection*{Compact generation}
An object $C$ in a triangulated category $\sfT$ admitting set-indexed
coproducts is \emph{compact} if the functor $\Hom_{\sfT}(C,-)$
commutes with all coproducts. We write $\sfT^{\sfc}$ for the full
subcategory of compact objects in $\sfT$. The category $\sfT$ is
\emph{compactly generated} if it is generated by a set of compact
objects.

Recall that we write $\Si$ for the suspension on $\sfT$. For objects
$X$ and $Y$ in $\sfT$, let
\[ 
\Hom^*_\sfT(X,Y)=\bigoplus_{i\in\bbZ}\Hom_\sfT(X,\Si^i Y)
\]
be the graded abelian group of morphisms. Set
$\End^{*}_{\sfT}(X)=\Hom^{*}_{\sfT}(X,X)$; this is a graded ring, and
$\Hom^*_\sfT(X,Y)$ is a right $\End^{*}_{\sfT}(X)$ and
left $\End^{*}_{\sfT}(Y)$-bimodule.

\subsection*{Central ring actions}
Let $R$ be a graded-commutative ring; thus $R$ is $\bbZ$-graded and
satisfies $rs=(-1)^{|r||s|}sr$ for each pair of homogeneous elements
$r,s$ in $R$.  We say that a triangulated category $\sfT$ is
\emph{$R$-linear}, or that $R$ \emph{acts} on $\sfT$, if there is a
homomorphism $\phi\col R\to Z^*(\sfT)$ of graded rings, where $
Z^*(\sfT)$ is the graded center of $\sfT$. This yields for each object
$X$ a homomorphism $\phi_X\col R\to\End^*_\sfT(X)$ of graded rings
such that for all objects $X,Y\in\sfT$ the $R$-module structures on
$\Hom^*_\sfT(X,Y)$ induced by $\phi_{X}$ and $\phi_{Y}$ agree, up to
the usual sign rule.

\bigskip

\emph{Henceforth $\sfT$ will be a compactly generated triangulated
  category with set-indexed coproducts, and $R$ a graded-commutative
  noetherian ring acting on $\sfT$.}

\bigskip

Since $\sfT$ is compactly generated with set-indexed coproducts, it
follows from the Brown representability theorem that $\sfT$ also
admits set-indexed products; see
\cite[Proposition~8.4.6]{Neeman:2001a}. This fact is used without
further comment.

\subsection*{Local cohomology and support} 
We write $\Spec R$ for the set of homogeneous prime ideals of $R$. Fix
$\fp\in\Spec R$ and let $M$ be a graded $R$-module. The homogeneous
localization of $M$ at $\fp$ is denoted by $M_{\fp}$ and $M$
is called \emph{$\fp$-local} when the natural map $M\to M_{\fp}$ is bijective.

Given a homogeneous ideal $\fa$ in $R$, we set
\[
\mcV(\fa) = \{\fp\in\Spec R\mid \fp\supseteq \fa\}\,.
\] 
A graded $R$-module $M$ is \emph{$\fa$-torsion} if each element of $M$
is annihilated by a power of $\fa$; equivalently, if
$M_\fp=0$ for all $\fp\in\Spec R \setminus\mcV(\fa)$.

The \emph{specialization closure} of a subset $\mcU$ of $\Spec R$ is
the set
\[
\cl\mcU=\{\fp\in\Spec R\mid\text{there exists $\fq\in\mcU$ with
  $\fq\subseteq \fp$}\}.
\] 
The subset $\mcU$ is \emph{specialization closed} if $\cl\mcU=\mcU$;
equivalently, if $\mcU$ is a union of Zariski closed subsets of $\Spec
R$. For each specialization closed subset $\mcV$ of $\Spec R$, we
define the full subcategory of $\sfT$ of \emph{$\mcV$-torsion objects}
as follows:
\[
\sfT_\mcV= \{X\in\sfT\mid\Hom^*_\sfT(C,X)_\fp= 0\text{ for all }
C\in\sfT^c,\, \fp\in \Spec R\setminus \mcV\}.
\]
This is a localizing subcategory and there exists a localization
functor $L_\mcV\col\sfT\to\sfT$ such that $\Ker L_\mcV=\sfT_\mcV$; see
\cite[Lemma 4.3, Proposition 4.5]{Benson/Iyengar/Krause:2008a}.

The localization functor $L_{\mcV}$ induces a colocalization functor
on $\sfT$, which we denote $\gam_\mcV$, and call the \emph{local
  cohomology functor} with respect to $\mcV$; see
Section~\ref{se:localizations}.  For each object $X$ in $\sfT$ there
is then an exact localization triangle
\begin{equation*}
\label{eq:locseq}
\gam_{\mcV}X\lto X\lto \bloc_{\mcV}X\lto\,.
\end{equation*}
In \cite{Benson/Iyengar/Krause:2008a} we established a number of
properties of these functors; for instance, that they commute
with all coproducts in $\sfT$, see
\cite[Corollary~6.5]{Benson/Iyengar/Krause:2008a}.

For each $\fp$ in $\Spec R$ and each object $X$ in $\sfT$ set
\[
X_\fp=L_{\mcZ(\fp)}X\,, \quad\text{where $\mcZ(\fp) = \{\fq\in\Spec
  R\mid \fq\not\subseteq \fp\}$.}
\]
The notation is justified by the fact that, by \cite[Theorem
  4.7]{Benson/Iyengar/Krause:2008a}, the adjunction morphism $X\to
X_{\fp}$ induces for any compact object $C$ an isomorphism of
$R$-modules
\[
\Hom_\sfT^*(C,X)_{\fp}\xra{\sim}\Hom_\sfT^*(C,X_{\fp})\,.
\] 
We say $X$ is \emph{$\fp$-local} if the adjunction morphism $X\to
X_\fp$ is an isomorphism; this is equivalent to the condition that
there exists \emph{some} isomorphism $X\cong X_{\fp}$ in $\sfT$.

Consider the exact functor $\gam_{\fp}\col\sfT\to\sfT$ obtained by
setting
\[
\gam_{\fp}X= \gam_{\mcV(\fp)}(X_{\fp}) \quad\text{for each object $X$
  in $\sfT$},
\]
and let $\gam_\fp\sfT$ denote its essential image.  One has a natural
isomorphism $\gam_{\fp}^{2}\cong \gam_{\fp}$, and an object $X$ from
$\sfT$ is in $\gam_\fp\sfT$ if and only if the $R$-module
$\Hom^{*}_{\sfT}(C,X)$ is $\fp$-local and $\fp$-torsion for every
compact object $C$; see
\cite[Corollary~4.10]{Benson/Iyengar/Krause:2008a}.

The \emph{support} of an object $X$ in $\sfT$ is by definition the set
\[
\supp_{R} X=\{\fp\in\Spec R\mid\gam_{\fp}X\ne 0\}.
\] 
One has $\supp_{R}X=\varnothing$ if and only if $X=0$ holds; see
\cite[Theorem~5.2]{\bik:2008a}.

\subsection*{Koszul objects}
For each object $X$ in $\sfT$ and each homogeneous ideal $\fa$ in $R$,
we denote $\kos X\fa$ a Koszul object on a finite sequence of elements
generating the ideal $\fa$; see \cite[\S5]{\bik:2008a}. Its
construction depends on a choice of a generating sequence, but the
localizing subcategory generated by it is independent of choice, and
depends only on the radical ideal of $\fa$; this follows from
\cite[Proposition~2.11(2)]{\bik:2009a}. Set
\[
X(\fp)=(\kos X \fp)_\fp\quad\text{for each $\fp\in\Spec R$.}
\]
The following computations will be used often:
\begin{equation}
\label{eq:kos-supp}
\supp_{R}(\kos X\fp) = \mcV(\fp)\cap \supp_{R}X\quad\text{and}\quad
\supp_{R}X(\fp) = \{\fp\}\cap \supp_{R}X
\end{equation}
For the first one, see \cite[Lemma 2.6]{\bik:2009a}; the second
follows, given \cite[Theorem~5.6]{\bik:2008a}.

The first part of the result below is \cite[Theorem~6.4]{\bik:2008a},
see also \cite[Proposition~2.7]{\bik:2009a}; the second one is part of
\cite[Proposition~3.9]{\bik:2009a}.

\begin{theorem}
\label{th:compactgeneration}
Suppose $\sfG$ is a set of compact generators for $\sfT$.  For each
specialization closed subset $\mcV$ and $\fp\in\Spec R$, there are
equalities
\[
\sfT_{\mcV} =\Loc_{\sfT}\big(\kos C{\fp}\mid C\in \sfG\text{ and
}\fp\in\mcV\big)\quad\text{and}\quad\gam_\fp\sfT=\Loc_\sfT(C(\fp)\mid C\in \sfG),\]
where both generating sets consist of compact objects.\qed
\end{theorem}

\section{Local homology and cosupport}
\label{se:cosupport}
Let $\sfT$ denote a compactly generated $R$-linear triangulated
category, as in Section~\ref{se:support}.  We introduce local homology
functors and a notion of cosupport for $\sfT$.

\subsection*{Local homology} 
Fix a specialization closed subset $\mcV\subseteq\Spec R$.  The
functors $L_\mcV$ and $\gam_\mcV$ on $\sfT$ preserve coproducts by
\cite[Corollary 6.5]{Benson/Iyengar/Krause:2008a} and hence have right
adjoints, by Brown representability. Following the notation in
Proposition~\ref{pr:coloc-fun}, this yields adjoint pairs
$(L_\mcV,V^\mcV)$ and $(\gam_\mcV,\lam^\mcV)$, and, for each
$X\in\sfT$, an exact triangle
\begin{equation}
\label{eq:loctriangle}
V^\mcV X\lto X\lto\lam^\mcV X\lto
\end{equation} 
We call $\lam^\mcV $ the \emph{local homology functor} with respect to
$\mcV$; see Remark~\ref{rem:gm}.

The commutation rules for the functors $L_{\mcV}$ and $\gam_{\mcV}$
given in \cite[Proposition 6.1]{\bik:2008a} carry over to their right
adjoints: For any specialization closed subset $\mcW$ of $\Spec R$
there are isomorphisms:
\begin{equation}
\label{eq:commutes}
\begin{gathered}
\lam^{\mcV}\lam^{\mcW} \cong \lam^{\mcV\cap\mcW} \cong
\lam^{\mcW}\lam^{\mcV}\\ V^{\mcV}V^{\mcW} \cong V^{\mcV\cup\mcW} \cong
V^{\mcW} V^{\mcV}\\ V^{\mcV}\lam^{\mcW} \cong \lam^{\mcW} V^{\mcV}
\end{gathered}
\end{equation}
If $\mcV\supseteq \mcW$ holds, then these isomorphisms and
Proposition~\ref{pr:coloc-fun}(2) yield:
\begin{equation}
\label{eq:vanishing}
V^{\mcV}\lam^{\mcW} \cong 0 \cong \lam^{\mcW}V^{\mcV}\,.
\end{equation}
These facts will be used without comment. For each $\fp\in\Spec R$ set
\begin{equation*}
\label{eq:Lam-fp}
\lam^\fp=\lam^{\mcV(\fp)}V^{\mcZ(\fp)}\,.
\end{equation*}
Note that $\lam^{\fp}\cong V^{\mcZ(\fp)}\lam^{\mcV(\fp)}$; that
$\lam^{\fp}\cong (\lam^{\fp})^{2}$; and that $(\gam_\fp,\lam^\fp)$ is
an adjoint pair.

\subsection*{Cosupport} 
The \emph{cosupport} of an object $X$ in $\sfT$ is the set
\[ 
\cosupp_R X=\{\fp\in\Spec R\mid\lam^\fp X\ne 0\}.
\]

Cosupport can be computed using Koszul objects, recalled in
Section~\ref{se:support}.

\begin{proposition}
\label{pr:hom-cosupp}
For each object $X$ in $\sfT$ and $\fp\in \Spec R$, one has
\[
\fp\in\cosupp_RX\iff\Hom_\sfT(C(\fp),X)\ne 0 \text{ for some
}C\in\sfT\,.
\]
Moreover, the object $C$ can be chosen from any set of compact
generators for $\sfT$.
\end{proposition}

\begin{proof}
If $\lam^{\fp}X\ne 0$, then $\Hom_\sfT(\lam^{\fp}X,\lam^\fp X)\ne 0$,
and hence $\Hom_\sfT(\gam_\fp\lam^{\fp}X,X)\ne 0$, since
$(\gam_{\fp},\lam^{\fp})$ form an adjoint pair. Since
$\gam_\fp\lam^{\fp}X$ is evidently in $\gam_{\fp}T$, the last
condition implies $\Hom_\sfT(C(\fp),X)\ne 0$ for some compact object
$C$ which is part of a generating set for $\sfT$, by
Theorem~\ref{th:compactgeneration}.

Conversely, if $\Hom_\sfT(C(\fp),X)\ne 0$ for some $C\in\sfT$, then
since $\gam_{\fp}C(\fp)\cong C(\fp)$, by, for instance,
Theorem~\ref{th:compactgeneration}, one obtains that
\[
\Hom_\sfT(C(\fp),\lam^{\fp}X)\cong \Hom_\sfT(\gam_{\fp}C(\fp),X)\cong
\Hom_\sfT(C(\fp),X)\ne 0\,.
\]
Thus, $\lam^{\fp}X\ne 0$, that is to say, $\fp$ is in $\cosupp_{R}X$.
\end{proof}

A important property of cosupport is that it is non-empty for non-zero
objects. We deduce this result from the corresponding statement for
supports and the result above. For another perspective, see
Theorem~\ref{th:max}.

\begin{theorem}
\label{th:cosupp=0}
For any $X\in \sfT$, one has $\cosupp_R X=\varnothing$ if and only if
$X=0$.
\end{theorem}

\begin{proof}
Clearly, $X=0$ implies $\cosupp_{R}X=\varnothing$.

If $X\ne 0$, then $\supp_{R}X\ne \varnothing$, by \cite[Theorem
  5.2]{\bik:2008a}. Pick a prime $\fp$ in $\supp_{R}X$, maximal with
respect to inclusion. Then $\supp_R(\kos X\fp)=\{\fp\}$, by
\eqref{eq:kos-supp}, which implies, in particular, that $\kos X\fp\ne
0$; equivalently, $\Hom_{\sfT}(\kos X\fp,\kos X\fp)\ne 0$. Moreover,
$\kos X\fp$ is $\fp$-local, hence isomorphic to $X(\fp)$, which
explains the isomorphism below:
\[
\Hom_{\sfT}(X(\fp),\kos X\fp)\cong \Hom_{\sfT}(\kos X\fp,\kos X\fp)\ne
0\,.
\]
Since $\kos X\fp$ is in $\Thick_{\sfT}(X)$, by construction, it
follows that $\Hom_{\sfT}(X(\fp),X)\ne 0$.  Hence $\fp\in\cosupp_RX$,
by Proposition~\ref{pr:hom-cosupp}.
\end{proof}

Next we describe further basic properties of local homology functors
and cosupports. The one below is immediate from the exactness of the
functor $\lam^\fp$.

\begin{proposition}
\label{pr:triangle}
\pushQED{\qed} For each exact triangle $X\to Y\to Z\to \Si X$ in
$\sfT$, one has
\[
\cosupp_{R} Y\subseteq\cosupp_{R} X\cup\cosupp_{R}
Z\quad\text{and}\quad\cosupp_{R}\Si X=\cosupp_{R} X.  \qedhere
\]
\end{proposition} 

There is a more precise result for exact
triangles~\eqref{eq:loctriangle}.

\begin{proposition}
\label{pr:localization-cosupport}
Let $\mcV$ be a specialization closed subset of $\Spec R$. For each
$X$ in $\sfT$ the following equalities hold:
\begin{align*}
&\cosupp_{R}\lam^{\mcV}X = \mcV\, \cap\,
  \cosupp_{R}X\\ &\cosupp_{R}V^{\mcV}X = \big(\Spec R\setminus
  \mcV\big)\, \cap\, \cosupp_{R}X.
\end{align*}
\end{proposition} 
\begin{proof}
Fix $\fp$ in $\Spec R$.  If $\fp\in\mcV$, then
$\lam^{\mcV(\fp)}\lam^\mcV=\lam^{\mcV(\fp)}$, by \eqref{eq:commutes},
and $V^{\mcZ(\fp)}\lam^\mcV=0$ if $\fp\not\in\mcV$, by
\eqref{eq:vanishing}. Hence one gets that:
\[
\lam^\fp\lam^\mcV = \begin{cases} \lam^\fp & \text{if
    $\fp\in\mcV$,}\\ 0 & \text{otherwise.}
\end{cases}
\]
The identity for $\cosupp_{R}\lam^{\mcV}X$ follows. The proof of the
second one is similar.
\end{proof}

The preceding result and Theorem~\ref{th:cosupp=0} yield:

\begin{corollary}
\label{co:cosupp}
Let $\mcV\subseteq \Spec R$ be specialization closed and $X$ an object
in $\sfT$. The following conditions are equivalent:
\begin{enumerate}
\item $\cosupp_{R}X\subseteq\mcV$.
\item $X\in\Im\lam^\mcV$.
\item the natural map $X\to\lam^{\mcV}X$ is an isomorphism.\qed
\end{enumerate}
\end{corollary} 

This result above is complemented by:

\begin{corollary}
\label{co:cosupp2}
Let $\mcV\subseteq \Spec R$ be specialization closed and $X$ an object
in $\sfT$. The following conditions are equivalent:
\begin{enumerate}
\item $\cosupp_{R}X\subseteq\Spec R\setminus\mcV$.
\item[(1$'$)] $\supp_{R}X\subseteq \Spec R\setminus\mcV$.
\item $X\in\Im V^\mcV$.
\item[(2$'$)] $X\in\Im L_{\mcV}$.
\item The natural map $V^{\mcV}X\to X$ is an isomorphism.
\item[(3$'$)] The natural map $X\to L_{\mcV}X$ is an isomorphism.
\end{enumerate}
In particular, $\cosupp_{R}X\subseteq \Spec R\setminus \mcZ(\fp)$ if
and only if $X$ is $\fp$-local.
\end{corollary} 

\begin{proof}
The equivalence of (1), (2), and (3) follows from
Proposition~\ref{pr:localization-cosupport} and
Theorem~\ref{th:cosupp=0}, while the equivalence of (1$'$), (2$'$),
and (3$'$) is part of \cite[Corollary~5.7]{\bik:2008a}.  It remains to
note that (2) $\Lra$ (2$'$), by Proposition~\ref{pr:coloc-fun}(2).

The last assertion is (1) $\Lra$ (3$'$), applied to $\mcV=\mcZ(\fp)$.
\end{proof}

The next result is an analogue of \cite[Corollary~5.8]{\bik:2008a}.

\begin{corollary}
\label{co:orthogonality}
Let $X,Y$ be objects in $\sfT$. Then $\cosupp_{R}
X\cap\cl(\cosupp_{R} Y)=\varnothing$ implies $\Hom_\sfT^{*}(X,Y)=0$.
\end{corollary} 

\begin{proof}
Set $\mcV=\cl(\cosupp_{R}Y)$. Then $X$ is in $\Im V^\mcV$, by
Corollary\ref{co:cosupp2}, and $Y$ is in $\Im\lam^\mcV$, by
Corollary~\ref{co:cosupp}, so $\Hom_\sfT^{*}(X,Y)=0$, by
Proposition~\ref{pr:coloc-fun}(2).
\end{proof} 

The next goal is Theorem~\ref{th:max}; the following two results
prepare for its proof. The one below is extracted from
\cite[Lemma~5.11]{\bik:2008a}.

\begin{lemma}
\label{le:kos-hom}
Let $\fa$ be a homogeneous ideal in $R$. For any objects $X$ and $Y$
in $\sfT$, the following statements hold.
\begin{enumerate}
\item The $R$-modules $\Hom^*_\sfT(\kos X\fa,Y)$ and
  $\Hom^*_\sfT(X,\kos Y\fa)$ are $\fa$-torsion.
\item $\Hom^*_\sfT(X,Y)=0$ implies $\Hom^*_\sfT(X,\kos Y\fa)=0$ and
  the converse holds when the $R$-module $\Hom^*_\sfT(X,Y)=0$ is
  $\fa$-torsion.
\item $\Hom^*_\sfT(\kos X\fa,Y)=0$ if and only if $\Hom^*_\sfT(X,\kos
  Y\fa)=0$.\qed
\end{enumerate}
\end{lemma}

The equality below is a version of \eqref{eq:kos-supp} for cosupport.

\begin{lemma}
\label{le:kos-cosupp}
Let $\fa$ be a homogeneous ideal in $R$ and $X$ an object in
$\sfT$. Then
\[
\cosupp_R(\kos X\fa) =\mcV(\fa)\cap\cosupp_R X\,.
\]
\end{lemma}

\begin{proof}
From \eqref{eq:kos-supp} one gets an equality
\[
\supp_{R}(\kos {C(\fp)}{\fa}) = \mcV(\fa) \cap \{\fp\} \cap \supp_{R}C
\quad\text{for any $C\in \sfT$.}
\] 
If $\Hom^{*}(C(\fp),\kos X\fa)\ne 0$ for some $C\in \sfT$, then
$\Hom^{*}(\kos {C(\fp)}\fa,X)\ne 0$; this follows from
Lemma~\ref{le:kos-hom}(3). In particular, $\kos {C(\fp)}\fa\ne 0$, so
$\fp\in\mcV(\fa)$, by the equality above. One thus obtains from
Proposition~\ref{pr:hom-cosupp} that $\cosupp_{R}(\kos {X}{\fa})
\subseteq\mcV(\fa)$.

When $\fp\in\mcV(\fa)$ holds, it follows from
Lemma~\ref{le:kos-hom}(1), and the observation that $C(\fp)$ is
isomorphic to $\kos {C_{\fp}}\fp$, that the $R$-module
$\Hom^{*}_{\sfT}(C(\fp),X)$ is $\fp$-torsion, and hence also
$\fa$-torsion. Therefore, Lemma~\ref{le:kos-hom}(2) yields:
\[
\Hom^{*}_{\sfT}(C(\fp),X) \ne 0 \iff \Hom^{*}_{\sfT}(C(\fp),\kos
X\fa)\ne 0\,.
\]
The desired equality involving cosupport now follows from
Proposition~\ref{pr:hom-cosupp}.
\end{proof}

Given $\mcU\subseteq \Spec R$, we write $\max\mcU$ for the set of
elements $\fp\in\mcU$ such that $\fq\in\mcU$ and $\fq\supseteq\fp$
imply $\fq=\fp$. Recall that $\lam^\fp\sfT$ denotes the essential
image of $\lam^\fp$.

\begin{theorem}
\label{th:max}
For each object $X$ in $\sfT$ there is an equality:
\[
\max(\supp_RX)=\max(\cosupp_RX)\,.
\]
Moreover, $\cosupp_{R}X(\fp)\subseteq \{\fp\}$ and $X(\fp)\in\lam^\fp\sfT\cap\gam_\fp\sfT$, for each $\fp\in\Spec R$. 
\end{theorem}

\begin{proof}
We prove $\max(\supp_RX)\subseteq \cosupp_RX$ and
$\max(\cosupp_RX)\subseteq\supp_RX$; the first equality would then
follow.

Fix $\fp$ in $\max(\supp_RX)$. Then $\supp_R(\kos X\fp)=\{\fp\}$, by
\eqref{eq:kos-supp}, so $\kos X\fp$ is $\fp$-local; see
Corollary~\ref{co:cosupp2}. It is always $\fp$-torsion, so $\kos
X\fp=X(\fp)$, and then
\[
\Hom_{\sfT}^{*}(X(\fp),\kos X\fp)\cong \Hom_{\sfT}^{*}(\kos X\fp,\kos
X\fp)\ne 0\,.
\]
This implies that $\fp$ is in $\cosupp_R(\kos X\fp)$ by
Proposition~\ref{pr:hom-cosupp}, hence also that it is in
$\cosupp_RX$, by Lemma~\ref{le:kos-cosupp}.

If $\fp\in\max(\cosupp_RX)$, then $\cosupp_R(\kos X\fp)=\{\fp\}$, by
Lemma~\ref{le:kos-cosupp}. Therefore the object $\kos X\fp$ is
$\fp$-local, by Corollary~\ref{co:cosupp2}, and $\fp$-torsion, so
$\supp_R(\kos X\fp)=\{\fp\}$. It remains to recall \eqref{eq:kos-supp}
to conclude that $\fp\in\supp_RX$.

Finally, $\cosupp_{R}X(\fp)\subseteq \mcV(\fp)$ by Lemma~\ref{le:kos-cosupp}, since $X(\fp)\cong
\kos{X_{\fp}}{\fp}$. Thus the inclusion $\supp_{R}X(\fp)\subseteq \{\fp\}$ from \eqref{eq:kos-supp} implies the corresponding inclusion for cosupport. It then follows from Corollaries~\ref{co:cosupp} and \ref{co:cosupp2} that $X(\fp)$ is in $\lam^{\fp}\sfT$. On the other hand, $X(\fp)$ is also in $\gam_\fp\sfT$, by~\cite[Corollary 5.7]{\bik:2008a}.
\end{proof}

To round off this material, we prove an analogue of
Proposition~\ref{pr:hom-cosupp} for supports.

\begin{proposition}
\label{pr:hom-supp}
For each object $X$ in $\sfT$ and $\fp\in \Spec R$, one has
\[
\fp\in\supp_RX\iff\Hom_\sfT(X,Y(\fp))\ne 0 \text{ for some
}Y\in\sfT\,.
\]
\end{proposition}

\begin{proof}
When $\fp$ is in $\supp_RX$ it follows from \eqref{eq:kos-supp} that
$X(\fp)\ne 0$, hence
\[
\Hom^*_\sfT(\kos{X}\fp,X(\fp)) \cong \Hom^*_\sfT(X(\fp),X(\fp))\ne
0\,;
\]
where the first isomorphism holds as $X(\fp)$ is $\fp$-local. Since
$\kos X\fp$ is in $\Thick_{\sfT}(X)$, one obtains that
$\Hom_\sfT(X,X(\fp))\ne 0$. This settles one implication.

The other implication follows from the chain of isomorphisms
\[
\Hom_\sfT(X,Y(\fp))\cong\Hom_\sfT(X,\lam^\fp
Y(\fp))\cong\Hom_\sfT(\gam_\fp X,Y(\fp)),
\]
where the first one holds because $\lam^\fp Y(\fp)\cong Y(\fp)$, by
Theorem~\ref{th:max}.
\end{proof}

\subsection*{Discrete sets}
Recall that a subset $\mcU\subseteq\Spec R$ is \emph{discrete} if $\fp
\subseteq\fq$ implies $\fp=\fq$ for each pair of primes
$\fp,\fq\in\mcU$.

\begin{proposition}
\label{prop:discrete}
Let $X$ and $Y$ be objects of $\sfT$ and $\mcU$ a discrete subset of
$\Spec R$. When $\supp_RX\subseteq \mcU$ and $\cosupp_RY\subseteq
\mcU$ there are natural isomorphisms
\[
X\xla{\sim} \coprod_{\fp\in\mcU}\gam_{\mcV(\fp)} X\xra{\sim}
\coprod_{\fp\in\mcU}\gam_\fp X \quad\text{and}\quad Y\xra{\sim}
\prod_{\fp\in\mcU}\lam^{\mcV(\fp)} Y\xla{\sim}
\prod_{\fp\in\mcU}\lam^\fp Y\,.
\]
\end{proposition}

\begin{proof}
The statement for $X$ is proved in \cite[Proposition~3.3]{\bik:2009a};
see also \cite[Theorem~7.1]{\bik:2008a}. Modifying the arguments by
taking adjunctions, and taking into account
Proposition~\ref{pr:product} below, yields the proof of the statement
for $Y$.
\end{proof}

Analogues of the next statement hold for $V^\mcV$ and $\lam^\mcV$
also.

\begin{proposition}
\label{pr:product}
Given objects $\{X_i\}_{i\in I}$ in $\sfT$ there is a natural
isomorphism
\[ 
\lam^\fp\left(\prod_{i\in I} X_i\right)\xra{\sim}\prod_{i\in
  I}\lam^\fp X_i.
\]
In particular, for any subset $\mcU\subseteq\Spec R$ the full
subcategory with objects
\[
\{X\in \sfT\mid \cosupp_{R}X\subseteq \mcU\}
\]
is a colocalizing subcategory of $\sfT$.
\end{proposition}

\begin{proof}
Right adjoints distribute over products. This applies to $\lam^\fp$,
which is right adjoint to $\gam_{\fp}$, to yield the desired
isomorphism. Given this, the second part of the statement follows, for
the subcategory in question equals
$\bigcap_{\fp\not\in\mcU}\Ker\lam^\fp$.
\end{proof}

\subsection*{Commutative noetherian rings}

Let $A$ be a commutative noetherian ring and $\sfD(A)$ the derived
category of the category of all $A$-modules. The category $\sfD(A)$ is
triangulated and compactly generated; indeed, $A$ is a compact
generator. It is also $A$-linear, where for each $M\in\sfD(A)$, the
homomorphism $A\to \Hom_{\sfD}(M,M)$ is given by scalar
multiplication.

\begin{remark}
\label{rem:gm}
Fix an ideal $\fa$ in $A$. In \cite[Theorem~9.1]{\bik:2008a} it is
proved that $\gam_{\mcV(\fa)}$ is the derived functor of the
$\fa$-torsion functor, which assigns an $A$-module $M$ to the module
$\varinjlim\Hom_{A}(A/\fa^{n},M)$. Greenlees and
May~\cite[\S2]{Greenlees/May:1992a}, see also
Lipman~\cite[\S4]{Lipman:2002a}, proved that the right adjoint of the
latter is local homology and that it coincides with left derived
functor of the $\fa$-adic completion functor, which assigns $M$ to
$\varprojlim M/\fa^n M$. In commutative algebra literature (as in
\cite{Lipman:2002a}), the $\fa$-adic completion functor itself would
usually be denoted $\lam_{\fa}$.  Our choice of notation,
$\lam^{\mcV(\fa)}$ for the right adjoint of $\gam_{\mcV(\fa)}$ is
based on this connection.
\end{remark}
 
The functor $V^{\mcV}$ too has, in this context, a familiar avatar, at
least in the special case $\mcV=\mcZ(\fp)$ for some $\fp$ in $\Spec
A$: The morphism $M\to M_{\fp}=\bloc_{\mcZ(\fp)}M$ is the usual
localization map, see the proof of \cite[Theorem~9.1]{\bik:2008a}, so
it follows from the classical Hom-tensor adjunction isomorphism that
its right adjoint is
\[
V^{\mcZ(\fp)}M \cong \RHom_{A}(A_{\fp},M)\,,
\]
and the morphism $\RHom_{A}(A_{\fp},M)\to M$ is the one induced by
$A\to A_{\fp}$. Hence
\[
\cosupp_{A}M=\{\fp\in \Spec A\mid
\RHom_{A}(A_{\fp},\lam^{\mcV(\fp)}M)\ne 0\}\,.
\]
In practice however, the cosupport seems hard to compute, even for
$M=A$. What little we know about this is contained in the following
results.

\begin{proposition}
\label{pr:cosuppZ}
For each $M$ in $\sfD(\bbZ)$ with $\hh M$ finitely generated, one has
\[
\supp_\bbZ M= \{\fp\in\Spec\bbZ\mid (\hh M)_\fp\ne 0\}= \cosupp_\bbZ M.
\]
\end{proposition}

\begin{proof}
The equality on the left is easily verified. For the one on the right,
given Theorem~\ref{th:max}, it suffices to prove that the zero ideal
is in $\cosupp_\bbZ M$ if and only if it is also in
$\supp_{\bbZ}M$. By Proposition~\ref{pr:hom-cosupp}, this amounts to
verifying that
\[
\RHom_{\bbZ}(\bbQ,M)\ne 0\iff {\bbQ\lotimes_{\bbZ}M}\ne 0\,.
\]
Since $M\cong \hh M$ in $\sfD(\bbZ)$ it suffices to verify the
equivalence above when $M$ is an indecomposable finitely generated
$\bbZ$-module, hence of the form $\bbZ/n\bbZ$, for some $n\geq 0$.
When $n\geq 1$, one has
$\RHom_{\bbZ}(\bbQ,\bbZ/n\bbZ)=0={\bbQ\lotimes_{\bbZ}\bbZ/n\bbZ}$. It
remains to verify that $\RHom_{\bbZ}(\bbQ,\bbZ)\ne 0$; see \cite[\S51,
  Exercise~7]{Fuchs:1970}.
\end{proof}

For complete local rings on the other hand, the difference between
cosupport and support can be as large as Theorem~\ref{th:max} permits.

\begin{proposition}
\label{pr:cosuppA}
Let $A$ be a commutative noetherian ring and $\fa$ an ideal in
$A$. The following conditions are equivalent:
\begin{enumerate}
\item
$A$ is $\fa$-adically complete.
\item
$\cosupp_{A}A\subseteq \mcV(\fa)$ holds.
\item
$\cosupp_{A}M\subseteq \mcV(\fa)$ holds for any $M\in\sfD(A)$ with $\hh M$ finitely generated.
\end{enumerate}
\end{proposition}

\begin{proof}
The equivalence of (1) and (2) is contained in Corollary~\ref{co:cosupp}. Thus it remains to prove that (1) implies
(3).

When $A$ is $\fa$-adically complete so is any finitely generated module. Since the $\fa$-adic completion functor is exact on finitely generated modules, one obtains an isomorphism $M\xra{\sim}\lam^{\mcV(\fa)}M$ for any such module $M$, and hence also for any complex $M$ with $\hh M$ finitely generated. Now apply Corollary~\ref{co:cosupp}.
\end{proof}

For any object $X$ in $\sfT$ and $\fp\in\Spec R$, the support of $X_{\fp}$ is contained in that of $X$; see \cite[Theorem 5.6]{\bik:2008a}. The corresponding statement for cosupports does not hold, which speaks to one of the difficulties in computing this invariant.

\begin{example}
Let $(A,\fm)$ be a local ring that is $\fm$-adically complete. Then for any $\fp\in\Spec A\setminus\{\fm\}$ one has
\[
\{\fp\}\subseteq \cosupp_{A}(A_{\fp})\not\subseteq \cosupp_{A}A=\{\fm\}\,.
\]
Indeed, the equality is by Proposition~\ref{pr:cosuppA}. That $\fp$ is in $\cosupp_{A}(A_{\fp})$ can be checked directly, or via Theorem~\ref{th:max} and the equality $\supp_{A}(A_{\fp})=\Spec A_{\fp}$.
\end{example}

\begin{remark}
The notion of cosupport for an $A$-module $M$ is not the same as the
one introduced by Richardson \cite{Richardson:2006a}, which is denoted
$\mathrm{coSupp}\, M$. For instance, it follows from
\cite[Theorem~2.7(vi)]{Richardson:2006a} that
$\mathrm{coSupp}\,\bbZ=\Spec\bbZ\setminus\{(0)\}$, while
Proposition~\ref{pr:cosuppZ} yields $\cosupp_{\bbZ}\bbZ=\Spec
\bbZ$. When $(A,\fm)$ is local, each non-zero finitely generated
module $M$ satisfies $\mathrm{coSupp}\,M=\{\fm\}$, by
\cite[Theorem~2.7(i)]{Richardson:2006a}, while
Proposition~\ref{pr:cosuppA} implies $\cosupp_{A}M=\{\fm\}$ if and only
if $M$ is $\fm$-adically complete.
\end{remark}

\section{$\fp$-local and $\fp$-complete objects}
\label{se:plocalpcomplete}
As before, let $\sfT$ denote a compactly generated $R$-linear
triangulated category. In this section we investigate the
subcategories $\lam^{\fp}\sfT$, for each $\fp\in\Spec R$. They contain
local information about $\sfT$ and are important in classifying
its colocalizing subcategories; see the discussion in the last part of
this section.

We begin by noting that an object $X$ is in $\lam^{\fp}\sfT$ if and
only if $\cosupp_{R}X\subseteq\{\fp\}$, if and only if $X$ is
$\fp$-local (that is, $X\to X_{\fp}$ is an isomorphism) and also
$\fp$-complete, meaning that the natural map $X\to \lam^{\mcV(\fp)}X$
is an isomorphism; see Corollaries~\ref{co:cosupp2} and
\ref{co:cosupp}. Also $\lam^{\fp}\sfT$ is a colocalizing subcategory
of $\sfT$, by Proposition~\ref{pr:product}.

Recall that $X$ is in $\gam_{\fp}\sfT$ if and only if
$\supp_{R}X\subseteq\{\fp\}$, if and only if $X$ is $\fp$-local and
$\fp$-torsion; see \cite[Corollaries~4.9, 5.10]{\bik:2008a}, and that
$\gam_{\fp}\sfT$ is a localizing subcategory.

\subsection*{Dwyer Greenlees correspondence}
Fix a prime $\fp$ in $\Spec R$. The result below, which establishes an
equivalence between the category of $\fp$-local and $\fp$-complete
objects and the category of $\fp$-local and $\fp$-torsion objects, may
be viewed as extension of such an equivalence discovered by Dwyer and
Greenlees~\cite[Theorem~2.1]{Dwyer/Greenlees:2002a}, see also Hovey,
Palmieri, and Strickland~\cite[Theorem
  3.3.5]{Hovey/Palmieri/Strickland:1997a}, to our setting.

\begin{proposition}
\label{pr:loc-complete}
The functors $\gam_\fp\col\lam^\fp\sfT\to\gam_\fp\sfT$ and
$\lam^\fp\col\gam_\fp\sfT\to\lam^\fp\sfT$ form an adjoint pair, that
are mutually quasi-inverse to each other.
\end{proposition}

\begin{proof}
Apply the identities in Proposition~\ref{pr:coloc-fun}.
\end{proof}

As a triangulated category, $\gam_\fp\sfT$ is compactly generated by
$\{C(\fp)\mid C\in\sfT^c\}$; see
Theorem~\ref{th:compactgeneration}. Given Theorem~\ref{th:max} and the
equivalence $\gam_\fp\sfT\xra{\sim}\lam^\fp\sfT$ in
Proposition~\ref{pr:loc-complete}, it follows that the same set
generates $\lam^\fp\sfT$, where the coproduct in $\lam^\fp\sfT$ is the
one induced from $\gam_{\fp}\sfT$.

Next we describe a set of cogenerators for $\lam^{\fp}\sfT$.

\subsection*{Perfect cogeneration}
Let $\sfU$ be a triangulated category with set-indexed products. A set
of objects $\sfS$ \emph{perfectly cogenerates} $\sfU$ if the following
conditions hold:
\begin{enumerate}
\item If $X$ is an object in $\sfU$ and $\Hom_\sfU(X,S)=0$ for all
  $S\in\sfS$ then $X=0$.
\item If a countable family of maps $X_i\to Y_i$ in $\sfU$ is such
  that
\[ 
\Hom_\sfU(Y_i,S)\to\Hom_\sfU(X_i,S)
\] 
is surjective for all $i$ and all $S\in\sfS$, then so are the induced
maps:
\[ 
\Hom_\sfU(\prod_iY_i,S)\to\Hom_\sfU(\prod_iX_i,S)\,.
\]
\end{enumerate}

\begin{proposition}
\label{pr:K}
If a set of objects $\sfS$ perfectly cogenerates $\sfU$ then
$\Coloc_\sfU(\sfS)=\sfU$.
\end{proposition}

\begin{proof}
The proof is akin to that of the corollary in
\cite[\S1]{Krause:2002a}.
\end{proof}

\subsection*{Injective objects}
Let now $\sfT$ be an $R$-linear triangulated category as before. For
each compact object $C$ in $\sfT$ and each injective $R$-module $I$,
Brown representability yields an object $T_{C}(I)$ in $\sfT$ and a
natural isomorphism:
\begin{equation}
\label{eq:TI}
\Hom_{\sfT}(-,T_{C}(I)) \cong \Hom_R(\Hom^*_\sfT(C,-),I).
\end{equation}

For $\fp\in\Spec R$, let $I(\fp)$ denote the injective envelope of
$R/\fp$.

\begin{proposition}
\label{pr:perf-cogen}
Fix $\fp\in\Spec R$. For each compact object $C$ in $\sfT$, one has
\[
\cosupp_R T_C(I(\fp))=\supp_RC\cap\{\fq\in\Spec R\mid\fq\subseteq\fp\}\,.
\]
Moreover, the set $\{ T_{\kos C\fp}(I(\fp))\mid C\in\sfT^c\}$ perfectly cogenerates $\lam^\fp\sfT$.
\end{proposition}

\begin{proof}
The shifts of $I(\fp)$ form a set of injective cogenerators for the
category of $\fp$-local $R$-modules.  For any $\fq\in\Spec R$ and
object $D$ in $\sfT$, one has equivalences
\begin{align*}
\Hom^*_{\sfT}(D(\fq),T_C(I(\fp)))\ne 0 &\iff
\Hom_{R}^*(\Hom_{\sfT}^{*}(C,D(\fq)),I(\fp))\ne 0 \\ &\iff
\Hom_{\sfT}^{*}(C,D(\fq))\ne 0 \text{ and } \fq\subseteq \fp\,.
\end{align*}
The first one is by \eqref{eq:TI}; the second holds as the $R$-module
$\Hom_{\sfT}^{*}(C,D(\fq))$ is $\fq$-local and $\fq$-torsion, by
\eqref{eq:kos-supp}. Propositions~\ref{pr:hom-cosupp} and
\ref{pr:hom-supp} now yield the stated equality.

Let $X$ be a non-zero object in $\lam^\fp\sfT$, so that
$\cosupp_{R}X=\{\fp\}$, and pick a compact object $C$ with
$\Hom^*_\sfT( C(\fp),X)\neq 0$; see
Proposition~\ref{pr:hom-cosupp}. It then follows that:
\begin{align*}
\Hom^*_\sfT(\kos C\fp,X)&\cong\Hom^*_\sfT(\kos C\fp,\lam^\fp
X)\\ &\cong\Hom^*_\sfT(\gam_\fp (\kos C\fp), X)\\ &\cong\Hom^*_\sfT(
C(\fp),X) \\ &\neq 0
\end{align*}
Replacing $C$ by an appropriate suspension $\Si^n C$, if necessary,
from the computation above and \eqref{eq:TI} one gets
\[
\Hom_{\sfT}(X,T_{\kos C\fp}(I(\fp)))\cong \Hom_R(\Hom^*_\sfT(\kos
C\fp,X),I(\fp))\neq 0\,.
\]
The other condition for perfect cogeneration holds because, for any
map $X\to Y$ in $\lam^\fp\sfT$, the induced map
\[
\Hom^*_\sfT(Y,T_{\kos C\fp}(I(\fp)))\to\Hom^*_\sfT(X,T_{\kos
  C\fp}(I(\fp)))
\]
is surjective if and only if $\Hom_\sfT^*(\kos C\fp,X)\to
\Hom_\sfT^*(\kos C\fp,Y)$ is injective.
\end{proof}

\subsection*{Classifying colocalizing subcategories}
We say that the \emph{local-global principle} holds for colocalizing
subcategories of $\sfT$ if for each object $X$ in $\sfT$ there is an
equality
\[
\Coloc_\sfT(X)=\Coloc_{\sfT}(\{\lam^{\fp}X\mid \fp\in\Spec R\}).
\]
The corresponding notion for localizing subcategories is investigated
in \cite[\S3]{\bik:2009a}. The reformulation below of the local-global
principle is easy to prove.

\begin{lemma}
\label{lem:colocal-global} 
\pushQED{\qed} The local-global principle for colocalizing
subcategories is equivalent to the statement: For any $X\in\sfT$ and
any colocalizing subcategory $\sfS$ of $\sfT$, one has
\[
X\in \sfS \iff \lam^{\fp}X\in \sfS \text{ for each }\fp\in\Spec
R. \qedhere
\]
\end{lemma}

Colocalizing subcategories of $\sfT$ are related to subsets of $\Spec
R$ via maps
\[
\left\{
\begin{gathered}
\text{colocalizing}\\ \text{subcategories of $\sfT$}
\end{gathered}\; \right\} 
\xymatrix@C=3pc {\ar@<1ex>[r]^-{{\sigma}} & \ar@<1ex>[l]^-{{\tau}}}
\left\{
\begin{gathered}
  \text{families $(\sfS(\fp))_{\fp\in\Spec R}$ with
    $\sfS(\fp)$}\\ \text{a colocalizing subcategory of
    $\lam^{\fp}\sfT$}
\end{gathered}\;
\right\}
\] 
which are defined by ${\sigma}(\sfS)=(\sfS\cap\lam^{\fp}\sfT)$ and
${\tau}(\sfS(\fp))=\Coloc_{\sfT}\big(\sfS(\fp)\mid \fp\in \Spec
R\big)$.

\begin{proposition}
\label{prop:classification}
If the local-global principle for colocalizing subcategories of $\sfT$
holds, then the map ${\sigma}$ is bijective, with inverse ${\tau}$.
\end{proposition}

\begin{proof}
We use the fact that $\lam^{\fp}$ is an idempotent exact functor
preserving products.

First observe that for each colocalizing subcategory $\sfS$ of $\sfT$
there is an inclusion
\[
\sfS \cap\lam^{\fp}\sfT \subseteq \lam^{\fp}\sfS\quad\text{for each
}\fp\in\Spec R. \tag{$\ast$}
\]
We prove that ${\sigma}{\tau}$ is the identity, that is to say, that
for any family $(\sfS(\fp))_{\fp\in\Spec R}$ of colocalizing
subcategories with $S(\fp)\subseteq \lam^{\fp}\sfT$ the colocalizing
subcategory cogenerated by all the $S(\fp)$, call it $\sfS$, satisfies
\[
\sfS \cap \lam^{\fp}\sfT=\sfS(\fp)\quad\text{for each }\fp\in\Spec R.
\]
Note that $\lam^{\fp}\sfS = \sfS(\fp)$ holds, since
$\lam^{\fp}\lam^{\fq}=0$ when $\fp\ne\fq$. Therefore ($\ast$) yields
an inclusion $\sfS \cap \lam^{\fp}\sfT \subseteq \sfS(\fp)$. The
reverse inclusion is obvious.

For any localizing subcategory $\sfS$ of $\sfT$, the reformulation of
the local-global principle in Lemma~\ref{lem:colocal-global} gives
$\sfS =\Coloc_{\sfT}(\sfS\cap\lam^{\fp}\sfT\mid\fp\in\Spec R)$. Thus
$\tau\sigma=\id$.
\end{proof}

\begin{remark}
\label{rem:costratification}
In analogy with the notion of stratification for localizing
subcategories of $\sfT$ introduced in \cite[\S4]{\bik:2009a}, we say
that $\sfT$ is \emph{costratified} by $R$ if
\begin{itemize}
\item[\quad\rm(C1)] The local-global principle holds for colocalizing
  subcategories of $\sfT$;
\item[\quad\rm(C2)] For each $\fp\in\Spec R$, the colocalizing
  subcategory $\lam^{\fp}\sfT$ contains no proper non-zero colocalizing
  subcategories.
\end{itemize}
It is immediate from Proposition~\ref{prop:classification} that when
these conditions hold, the maps $\sigma$ and $\tau$ induce a bijection
\[ 
\left\{\begin{gathered} \text{Colocalizing}\\ \text{subcategories of
  $\sfT$}
\end{gathered}\;
\right\} \xymatrix@C=3pc{ \ar@<1ex>[r]^-{\cosupp_R} &
  \ar@<1ex>[l]^-{\cosupp_R^{-1}}} \left\{
\begin{gathered}
  \text{subsets of $\supp_R\sfT$}
\end{gathered}\;
\right\}.
\] 
For the main results of this work, it suffices to consider a version
of costratification for tensor triangulated categories, see
Section~\ref{se:costratification}; so we do not study the general
notion in any great detail.
\end{remark}

\section{Axioms for support and cosupport}

In this section we give an axiomatic description of cosupport,
analogous to the one for support in \cite[Theorem~5.15]{\bik:2008a};
see also Theorem~\ref{th:supp-axioms} below. This material is not used
elsewhere in this paper.

As before, we fix a compactly generated $R$-linear triangulated
category $\sfT$.  The starting point is the following cohomological
addendum to Corollary~\ref{co:cosupp2}.

\begin{lemma}
\label{le:cohomol}
Let $\mcV\subseteq\Spec R$ be a specialization closed subset. For each
object $X$ in $\sfT$, the following conditions are equivalent:
\begin{enumerate}
\item $\mcV\cap\cosupp_RX=\varnothing$.
\item[(1$'$)] $\mcV\cap\supp_RX=\varnothing$.
\item $\Hom^*_\sfT(\kos C\fp,X)=0$ for all $C\in\sfT^c$ and
  $\fp\in\mcV$.
\item $\Hom^*_\sfT(C,X)$ is either zero or not $\fp$-torsion for each
  $C\in\sfT^c$ and $\fp\in\mcV$.
\end{enumerate}
\end{lemma}

\begin{proof}
(1) $\Lra$ (1$'$) is part of Corollary~\ref{co:cosupp2}, while (1$'$)
  $\Lra$ (2) is a consequence of
  Theorem~\ref{th:compactgeneration}. To prove that (2) $\Lra$ (3),
  use Lemma~\ref{le:kos-hom}.
\end{proof}

\begin{remark}
As noted above, for each specialization closed subset $\mcV$ of $\Spec
R$ there is an equality of subcategories
\[
\{X\in\sfT\mid \cosupp_{R}X \cap \mcV=\varnothing\} = \{X\in\sfT\mid
\supp_{R}X \cap \mcV=\varnothing\}\,.
\]
The subcategory on the left is colocalizing, by
Proposition~\ref{pr:product}, while the one on the right is
localizing, since $\gam_{\fp}$ preserves set-indexed coproducts; see
\cite[Corollary~6.6]{\bik:2008a}.
\end{remark}

Given a specialization closed subset $\mcV\subseteq\Spec R$,
Lemma~\ref{le:cohomol} yields cohomological criteria for the
(co)support of any object $X$ to be contained in $\Spec
R\setminus\mcV$. This leads to axiomatic descriptions of cosupport and
support.

\begin{theorem}
\label{th:cosupp-axioms}
There exists a unique assignment sending each object $X$ in $\sfT$ to
a subset $\cosupp_{R} X$ of $\Spec R$ such that the following
properties hold:
\begin{enumerate}
\item \emph{Cohomology:} For each object $X$ in $\sfT$ and each $\fp$
  in $\Spec R$, one has
\[
\mcV(\fp)\cap\cosupp_RX\ne\varnothing
\] 
if and only if $\Hom_\sfT^*(C,X)$ is non-zero and $\fp$-torsion for
some $C$ in $\sfT^c$.
\item \emph{Orthogonality:} For objects $X$ and $Y$ in $\sfT$, one has
  that
\[
\cosupp_{R} X\cap\cl(\cosupp_{R} Y)=\varnothing
\quad\text{implies}\quad \Hom_\sfT(X,Y)=0.
\]
  \item \emph{Exactness:} For every exact triangle $X\to Y\to Z\to$ in
    $\sfT$, one has
\[
\cosupp_{R} Y\subseteq \cosupp_{R} X \cup \cosupp_{R} Z.
\] 
\item \emph{Separation:} For any specialization closed subset $\mcV$
  of $\Spec R$ and object $X$ in $\sfT$, there exists an exact
  triangle $X'\to X\to X''\to$ in $\sfT$ such that
\[
\cosupp_{R} X'\subseteq\Spec R\setminus\mcV \quad\text{and}
\quad\cosupp_{R}X'' \subseteq \mcV.
\]
\end{enumerate}
\end{theorem} 
\begin{proof}
Lemma~\ref{le:cohomol} implies (1). Corollary~\ref{co:orthogonality}
is (2), and Proposition~\ref{pr:triangle} is (3).
Proposition~\ref{pr:localization-cosupport} implies (4). Here one uses
for any specialization closed subset $\mcV$ of $\Spec R$ the
localization triangle \eqref{eq:loctriangle}.

Now let $\si\col \sfT\to \Spec R$ be a map satisfying properties
(1)--(4).

Fix a specialization closed subset $\mcV\subseteq\Spec R$ and an
object $X\in \sfT$.  It suffices to verify that the following
equalities hold:
\[
\si(\lam^{\mcV}X)= \si(X)\cap \mcV \quad\text{and}\quad
\si(V^{\mcV}X)= \si(X)\cap(\Spec R\setminus \mcV).\tag{$\ast$}
\]
Indeed, for any point $\fp$ in $\Spec R$ one then obtains that
\begin{align*}
\si(\lam^{\fp}X) &=\si(V^{\mcZ(\fp)}\lam^{\mcV(\fp)}{X})
\\ &=\si(\lam^{\mcV(\fp)}{X})\cap (\Spec R\setminus \mcZ(\fp))
\\ &=\si(X)\cap \mcV(\fp) \cap (\Spec R\setminus
\mcZ(\fp))\\ &=\si(X)\cap \{\fp\}.
\end{align*}
Therefore, $\fp\in \si(X)$ if and only if $\si(\lam^{\fp}X)\ne
\varnothing$; this last condition is equivalent to $\lam^{\fp}X\ne 0$,
by the cohomology property. The upshot is that $\fp\in \si(X)$ if and
only if $\fp\in \supp_{R}X$, which is the desired conclusion.

It thus remains to prove ($\ast$).

Let $X'\to X\to X''\to$ be the triangle associated to $\mcV$, provided
by property (4). It suffices to verify the following statements:
\begin{enumerate}
\item[(i)] $\si(X'')= \si(X)\cap \mcV$ and $\si(X')= \si(X)\cap(\Spec
  R\setminus \mcV)$;
\item[(ii)] $\lam^{\mcV}X\cong X'$ and $\bloc_{\mcV}X\cong X''$.
\end{enumerate} 

The equalities in (i) are immediate from properties (3) and (4). In
verifying (ii), the crucial observation is that, by the cohomology
property, for any $Y$ in $\sfT$ one has
\[
\mcV\,\cap\,\si(Y)=\varnothing\iff Y\in\Im V^\mcV\,.
\] 
Thus $X'$ is in $\Im V^\mcV$. On the other hand, property (2) and
Lemma~\ref{le:loc-basic} imply that $X''$ is in $\Im\lam^\mcV$.  One
thus obtains the following morphism of triangles
\[
\xymatrixrowsep{2pc} \xymatrixcolsep{2pc} \xymatrix{ X'\ar@{->}[r]
  \ar@{->}[d]^{\alpha} & X \ar@{=}[d]\ar@{->}[r]
  &X''\ar@{->}[d]^{\beta}\ar@{->}[r]& \\ V^{\mcV}{X}\ar@{->}[r] & X
  \ar@{->}[r] &\lam^{\mcV}{X}\ar@{->}[r]&}
\]
where the object $\Cone(\alpha)\cong\Cone(\Si^{-1}\beta)$ belongs to
$\Im V^\mcV\cap\Im\lam^\mcV$, hence is trivial. Therefore, $\alpha$
and $\beta$ are isomorphisms, which yields (ii).
\end{proof}

The following axiomatic description of support is a slight
modification of \cite[Theorem~5.15]{\bik:2008a}. Note that conditions
(1) and (3) coincide for support and cosupport. The crucial difference
appears in conditions (2) and (4).

\begin{theorem}
\label{th:supp-axioms}
There exists a unique assignment sending each object $X$ in $\sfT$ to
a subset $\supp_{R} X$ of $\Spec R$ such that the following properties
hold:
\begin{enumerate}
\item \emph{Cohomology:} For each object $X$ in $\sfT$ and each $\fp$
  in $\Spec R$, one has
\[
\mcV(\fp)\cap\supp_RX\ne\varnothing
\] 
if and only if $\Hom_\sfT^*(C,X)$ is non-zero and $\fp$-torsion for
some $C$ in $\sfT^c$.
\item \emph{Orthogonality:} For objects $X$ and $Y$ in $\sfT$, one has
  that
\[
\cl(\supp_{R} X)\cap\supp_{R} Y=\varnothing \quad\text{implies}\quad
\Hom_\sfT(X,Y)=0.
\]
  \item \emph{Exactness:} For every exact triangle $X\to Y\to Z\to$ in
    $\sfT$, one has
\[
\supp_{R} Y\subseteq \supp_{R} X \cup \supp_{R} Z.
\] 
\item \emph{Separation:} For any specialization closed subset $\mcV$
  of $\Spec R$ and object $X$ in $\sfT$, there exists an exact
  triangle $X'\to X\to X''\to$ in $\sfT$ such that
\[
\supp_{R} X'\subseteq\mcV \quad\text{and} \quad\supp_{R}X'' \subseteq
\Spec R\setminus\mcV.
\]
\end{enumerate}
\end{theorem} 
\begin{proof}
Adapt the proof of \cite[Theorem~5.15]{\bik:2008a}, using
Lemma~\ref{le:cohomol}.
\end{proof}

\section{Change of rings and categories}
\label{se:change of categories}

In this section we discuss how support and cosupport is affected by
the change of rings and categories.  Throughout $R$ is a
graded-commutative noetherian ring.

\subsection*{Linear functors}
Let $\sfT$ and $\sfU$ be $R$-linear triangulated categories. We say
that a functor $F\col \sfT\to \sfU$ is \emph{$R$-linear} if it is an
exact functor such that for each $X$ in $\sfT$ the following diagram
is commutative:
\[
\xymatrixrowsep{2pc} \xymatrixcolsep{1pc} \xymatrix{ &R
  \ar@{->}[dr]^{\phi_{FX}} \ar@{->}[dl]_{\phi_{X}} \\
\End_{\sfT}^*(X) \ar@{->}[rr]^-{F} & &\End_{\sfU}^*(FX)
}
\]

Let $F\col \sfT\to \sfU$ be an $R$-linear functor, and let $X$ and $Y$
be objects in $\sfT$ and $\sfU$, respectively. The structure
homomorphisms $\phi_X$ and $\phi_{FX}$ provide two $R$-module
structures on $\Hom^*_{\sfU}(FX,Y)$, and the $R$-linearity of $F$ is
equivalent to the claim that these coincide.

We are grateful to the referee for suggesting the following lemma.
\begin{lemma}
\label{lem:linearity}
Let  $F\col \sfT\to \sfU$ be an $R$-linear functor and $G$ a right adjoint. Then the following statements hold.
\begin{enumerate}
\item The adjunction isomorphism $\Hom^*_{\sfT}(X,GY)\xra{\sim}
  \Hom^*_{\sfU}(FX,GY)$ is $R$-linear.
\item The functor $G$ is $R$-linear.
\end{enumerate}
\end{lemma}
\begin{proof}
For (1), observe that the adjunction isomorphism can be factored as
\[
\Hom^*_{\sfT}(X,GY)\xra{F}
  \Hom^*_{\sfU}(FX,FGY)\xra{(FX,\theta Y)}
 \Hom^*_{\sfU}(FX,Y)
 \]
where the second map is induced by the counit $\theta\colon FG\to\Id_\sfU$. For (2), note that the map $\Hom_\sfU(Y,Y)\xra{G}\Hom_\sfT(GY,GY)$ can be factored as
\[
\Hom^*_{\sfU}(Y,Y)\xra{(\theta Y,Y)}
\Hom^*_{\sfU}(FGY,Y)\xra{\sim} \Hom^*_{\sfT}(GY,GY)
\] 
where the second map is the adjunction isomorphism.
\end{proof}

\begin{proposition}
\label{prop:lchchangecat}
Let $F\col\sfT\to \sfU$ be a functor between compactly generated
$R$-linear triangulated categories which preserves set-indexed
coproducts and products. Let $E$ be a left adjoint of $F$, and suppose
that $F$ or $E$ is $R$-linear.  For any specialization closed subset
$\mcV$ of $\Spec R$ there are then natural isomorphisms
\[
F\gam_{\mcV}\cong\gam_{\mcV}F,\quad
F\bloc_{\mcV}\cong\bloc_{\mcV}F,\quad \gam_{\mcV}E\cong
E\gam_{\mcV},\quad\text{and}\quad \bloc_{\mcV}E\cong E\bloc_{\mcV}\,.
\]
\end{proposition}

Note that the functor $F$ admits a left adjoint by Brown
representability, because $F$ preserves set-indexed products.

\begin{proof}
For each object $X$ in $\sfT$, one has an exact triangle
\[
F\gam_{\mcV}X \lto FX \lto F\bloc_{\mcV}X\lto
\]
induced by the localization triangle for $\mcV$. It thus suffices to
verify that $F\gam_{\mcV}X$ is in $\Im\gam_\mcV$ and that
$F\bloc_{\mcV}X$ is in $\Im\bloc_{\mcV}$; see
\cite[\S4]{Benson/Iyengar/Krause:2008a}.  We use the
adjunction isomorphisms
\begin{equation}\label{eq:adjoint}
\begin{split}
\Hom_{\sfU}^*(C,F\gam_{\mcV}X) &\cong
\Hom_{\sfT}^*(EC,\gam_{\mcV}X)\\ \Hom_{\sfU}^*(C,FL_{\mcV}X) &\cong
\Hom_{\sfT}^*(EC,L_{\mcV}X)
\end{split}
\end{equation}
which are $R$-linear because $F$ or $E$ is $R$-linear; see
Lemma~\ref{lem:linearity}. Note that $EC$ is compact if $C$ is
compact, since $F$ preserves set-indexed coproducts.

An object $Y$ in $\sfU$ belongs to $\Im\gam_\mcV$ if and only if
$\Hom_{\sfU}^*(C,Y)_\fp=0$ for all compact $C\in\sfU$ and all
$\fp\in\Spec R\setminus\mcV$.  Applying this characterization to
$F\gam_{\mcV}X$ and $\gam_\mcV X$, the adjunction
\eqref{eq:adjoint} implies that $F\gam_{\mcV}X$ is in $\Im\gam_\mcV$.

An object $Y$ in $\sfU$ belongs to $\Im L_\mcV$ if and only if
$\Hom_{\sfU}^*(\kos C\fp ,Y)=0$ for all compact $C\in\sfU$ and all
$\fp\in\mcV$, by Corollary~\ref{co:cosupp2} and
Lemma~\ref{le:cohomol}. Applying this characterization to $F
L_{\mcV}X$ and $L_\mcV X$, the adjunction \eqref{eq:adjoint} implies
that $F L_{\mcV}X$ is in $\Im L_\mcV$.  Here, one uses that $E(\kos
C\fp)\cong\kos{(EC)}\fp$, and this completes the proof of the first
pair of isomorphisms.

The isomorphism $F\bloc_{\mcZ(\fp)}\cong \bloc_{\mcZ(\fp)}F$ implies
$F(X(\fp))\cong (FX)(\fp)$ for each $X$ in $\sfT$ and each $\fp$ in
$\Spec R$. Given this, the proof of the isomorphisms involving $E$ is
similar: For each object $Y$ in $\sfU$, it follows from
Proposition~\ref{pr:hom-supp} and adjunction that there are inclusions
\[
\supp_R E\gam_\mcV Y\subseteq\mcV \quad\text{and}\quad \supp_R E
L_\mcV Y\subseteq \Spec R\setminus\mcV\,.
\]
Thus $E\gam_\mcV Y$ is in $\Im\gam_\mcV$ and $EL_\mcV Y$ in $\Im L_\mcV$.
\end{proof}

\begin{remark}
In the preceding proof, the assumption on $F$ or $E$ to be $R$-linear
is only used for the $R$-linearity of the adjunction isomorphisms
\eqref{eq:adjoint}.
\end{remark}

\subsection*{Change of rings}
Let $S$ be a graded-commutative noetherian ring, and $\sfU$ an
$S$-linear triangulated category.  Given a homomorphism of rings
$\alpha\col R\to S$, there is a natural $R$-linear structure on $\sfU$
induced by homomorphisms
\[
R\xra{\alpha} S\xra{\phi_{X}}\End^{*}_{\sfU}(X)\quad\text{for $X\in
  \sfU$}.
\]

As usual, $\alpha$ induces a map $\alpha^*\col \Spec S\to \Spec R$,
with $\alpha^{*}(\fq)=\alpha^{-1}(\fq)$ for each $\fq$ in $\Spec
S$. Observe that if $\mcV\subseteq\Spec R$ is specialization closed,
then so is the subset $(\alpha^*)^{-1}\mcV$ of $\Spec S$.

\begin{proposition}
\label{prop:lchchangerings}
Let $\alpha\col R\to S$ be a homomorphism of rings, and $\sfU$ an
$S$-linear triangulated category, with induced $R$-linear structure
via $\alpha$. Let $\mcV\subseteq\Spec R$ be a specialization closed
set and $\mcW=(\alpha^*)^{-1}\mcV$. Then there are isomorphisms
\[
\gam_{\mcV}\cong \gam_{\mcW}\quad\text{and}\quad \bloc_{\mcV}\cong
\bloc_{\mcW}.
\]
\end{proposition}

\begin{proof}
It suffices to prove that any object $X\in\sfU$ is in $\sfU_{\mcV}$ if
and only if it is in $\sfU_{\mcW}$. Thus one needs to show for any
compact object $C\in\sfU$ that $\Hom^{*}_{\sfU}(C,X)_\fp=0$ for all
$\fp\in\Spec R\setminus \mcV$ if and only if
$\Hom^{*}_{\sfU}(C,X)_\fq=0$ for all $\fq\in\Spec
S\setminus\mcW$. This one finds in Lemma~\ref{lem:rings} below.
\end{proof}

\begin{lemma}
\label{lem:rings}
Let $\alpha\col R\to S$ be a homomorphism of graded-commutative
noetherian rings and $\mcV\subseteq \Spec R$ a specialization closed
subset. Given an $S$-module $M$, one has $M_\fp=0$ for all
$\fp\in\Spec R\setminus \mcV$ if and only if $M_\fq=0$ for all
$\fq\in\Spec S\setminus (\alpha^*)^{-1}\mcV$.
\end{lemma}

\begin{proof}
Suppose first that $M_{\fp}=0$ for all $\fp\not\in\mcV$, and choose
$\fq\not\in (\alpha^*)^{-1}\mcV$. Then $M_{\fq}\cong (M_{\alpha^*(\fq)})_{\fq}=0$.

Assume now that $M_{\fp}\ne0$ for some $\fp\not\in\mcV$. We view
$M_\fp$ as an $S_{\fp}$-module and find therefore a prime ideal $\fq$
in $\Spec S_{\fp}\subseteq \Spec S$ such that $(M_{\fp})_{\fq}\cong
M_{\fq}$ is non-zero. It remains to observe that
$\alpha^*(\fq)\subseteq \fp$ and hence that $\alpha^{*}(\fq)$ is not
in $\mcV$.
\end{proof}

\subsection*{Change of rings and categories}
Henceforth, we say \emph{$(F;\alpha)\col (\sfT;R)\lto (\sfU;S)$ is an
  exact functor} to mean that $\sfT$ and $\sfU$ are compactly
generated $R$-linear and $S$-linear triangulated categories,
respectively; $\alpha\col R\to S$ is a homomorphism of graded rings;
and $F$ is an exact functor that is $R$-linear with respect to the
induced $R$-linear structure on $\sfU$; in other words, that the
diagram
\begin{equation*}
\label{eq:alpha}
\xymatrixrowsep{2pc} \xymatrixcolsep{2pc} \xymatrix{ R
  \ar@{->}[r]^{\alpha} \ar@{->}[d]_{\phi_X} & S
  \ar@{->}[d]^{\phi_{FX}} \\
\End_{\sfT}^*(X) \ar@{->}[r]^-{F} & \End_{\sfU}^*(FX)
}
\end{equation*}
is commutative for each $X\in \sfT$.

\begin{theorem}
\label{th:cochange}
Let $(F;\alpha)\col (\sfT; R) \to (\sfU;S)$ be an exact functor which
preserves set-indexed coproducts and products. Let $E$ be a left
adjoint and $G$ a right adjoint of $F$. Let $\mcV\subseteq\Spec R$ be
a specialization closed subset and set
$\mcW=(\alpha^*)^{-1}\mcV$. Then there are natural isomorphisms:
\begin{gather*}
\tag{1} F\gam_{\mcV}\cong \gam_{\mcW}F, \quad F\bloc_{\mcV}\cong
\bloc_{\mcW}F,\quad \gam_{\mcV}E\cong E \gam_{\mcW}, \quad \text{and}
\quad \bloc_{\mcV}E\cong E\bloc_{\mcW}\\ \tag{2} F\lam^{\mcV}\cong
\lam^{\mcW}F, \quad F V^{\mcV}\cong V^{\mcW}F,\quad
\lam^{\mcV}G\cong G \lam^{\mcW}, \quad \text{and}\quad V^{\mcV}G\cong
GV^{\mcW}\,.
\end{gather*}
\end{theorem}

This result contains Propositions~\ref{prop:lchchangecat} and
\ref{prop:lchchangerings}: to recover the first, set $\alpha=\id_{R}$;
for the second set $F=\Id_{\sfT}$. On the other hand, it is proved
using the latter results.

\begin{proof}
Since $F$ is linear with respect to the induced $R$-linear structure
on $\sfU$, Propositions~\ref{prop:lchchangecat} and
\ref{prop:lchchangerings} yield the following isomorphism:
\[
F\gam_{\mcV}\cong \gam_{\mcV}F \cong \gam_{\mcW}F.
\]
The other isomorphisms in (1) can be obtained in the same way.

The isomorphisms in (2) are obtained by taking right adjoints of those
in (1).
\end{proof}

As applications, we establish results which track the change in
support along linear functors; this is one reason we have had to
introduce these notions.

\begin{corollary}
\label{cor:basechange-support}
Let $(F;\alpha)\col (\sfT;R) \to (\sfU;S)$ be an exact functor which
preserves set-indexed coproducts and products.  Let $E$ be a left
adjoint and $G$ a right adjoint of $F$.  Then for $X\in\sfT$ and
$Y\in\sfU$ there are inclusions:
\begin{gather*}
\tag{1} \alpha^*(\supp_{S}FX)\subseteq \supp_{R}X\quad\text{and}\quad
\supp_{R}EY\subseteq\alpha^*(\supp_{S}Y)\\ \tag{2}\alpha^*(\cosupp_{S}FX)\subseteq
\cosupp_{R}X\quad\text{and}\quad
\cosupp_{R}GY\subseteq\alpha^*(\cosupp_{S}Y)\,.
\end{gather*}
Each inclusion is an equality when the corresponding functor is
faithful on objects.
\end{corollary}

\begin{proof}
Let $\fp$ be a point in $\Spec R$, and pick specialization closed
subsets $\mcV$ and $\mcW$ of $\Spec R$ such that
$\{\fp\}=\mcV\setminus\mcW$. For example, set $\mcV=\mcV(\fp)$ and
$\mcW=\mcV\setminus\{\fp\}$.

Setting $\wt\mcV=(\alpha^*)^{-1}\mcV$ and
$\wt\mcW=(\alpha^*)^{-1}\mcW$, one gets isomorphisms
\[
F(\gam_{\fp}X) =F\bloc_{\mcW}\gam_{\mcV}X =
\bloc_{\wt\mcW}\gam_{\wt\mcV}FX\,.
\]
Observing that $\wt\mcV\setminus\wt\mcW=(\alpha^*)^{-1}\{\fp\}$, this
yields equalities
\[
\supp_{S}F(\gam_{\fp}X) = \supp_{S}FX \cap (\wt\mcV\setminus \wt\mcW)
= \supp_{S}FX\cap (\alpha^*)^{-1}\{\fp\}\,.
\]
Thus, $\alpha^*(\supp_{S}FX)\subseteq \supp_{R}X$, and equality holds
if $F$ is faithful on objects.

The other inclusions are obtained in the same way.
\end{proof}

In the preceding result, the stronger conclusion
$\supp_{S}FX=(\alpha^*)^{-1}\supp_{R}X$ need not hold, even when $F$
is an equivalence of categories.

\begin{example}
\label{ex:prima}
Let $R$ be a field, set $S=R[a]/(a^2-a)$, and let $\sfU=\sfD(S)$
denote the derived category of $S$-modules, with canonical $S$-linear
structure. Let $\sfT=\sfU$ and view this as an $R$-linear triangulated
category via the inclusion $\alpha\col R\to S$.

Let $F\col\sfT\to \sfU$ be the identity functor; it is evidently
compatible with $\alpha$ and faithful.  Observe however that for the
module $X=S/(a)$ in $\sfT$ one has
\[
\supp_{R} X = \Spec R\quad\text{and}\quad \Spec_{S}FX = \{(a)\}.
\]
On the other hand, $(\alpha^*)^{-1}\supp_{R}X = \Spec S$.
\end{example}

\begin{corollary}
\label{cor:basechange-fibre-gen}
Let $(F;\alpha)\col (\sfT;R) \to (\sfU;S)$ be an exact functor which
preserves set-indexed coproducts and products.  Let $E$ be a left adjoint
and $G$ a right adjoint of $F$. Let $\fp\in\Spec R$ and suppose that
$\mcU=(\alpha^*)^{-1}\{\fp\}$ is a discrete subset of $\Spec S$. Then
there are isomorphisms:
\begin{gather*}
\tag{1} F\gam_\fp\cong \coprod_{\fq\in\mcU}
\gam_{\fq}F\quad\text{and}\quad \gam_\fp E\cong \coprod_{\fq\in\mcU}
E\gam_{\fq}\\ \tag{2} F\lam^\fp\cong \prod_{\fq\in\mcU}
\lam^{\fq}F\quad\text{and}\quad \lam^\fp G\cong \prod_{\fq\in\mcU}
G\lam^{\fq}\,.
\end{gather*}
\end{corollary}

\begin{proof}
Choose specialization closed subsets $\mcV$ and $\mcW$ of $\Spec R$
with $\mcV\setminus\mcW=\{\fp\}$. Setting
$\wt\mcV=(\alpha^*)^{-1}\mcV$ and $\wt\mcW=(\alpha^*)^{-1}\mcW$, one
gets isomorphisms
\[
F\gam_{\fp} =F\bloc_{\mcW}\gam_{\mcV}\cong
\bloc_{\wt\mcW}\gam_{\wt\mcV}F \cong \coprod_{\fq\in\mcU} \gam_{\fq}
\bloc_{\mcW}\gam_{\mcV}F \cong \coprod_{\fq\in\mcU} \gam_{\fq}F\,,
\]
where the first one follows Theorem~\ref{th:cochange}.  The second is
by Proposition~\ref{prop:discrete}, which applies since
$\wt\mcV\setminus\wt\mcW=\mcU$ and $\mcU$ is discrete. The last
isomorphism holds because one has
$\gam_\fq\bloc_{\wt\mcW}\gam_{\wt\mcV}\cong\gam_\fq$ for all
$\fq\in\mcU$; see \cite[Proposition~6.1]{Benson/Iyengar/Krause:2008a}.

The other isomorphisms can be obtained in the same way.
\end{proof}

The next result involves a notion of costratification of $R$-linear
triangulated categories. This has been introduced in
Remark~\ref{rem:costratification}, and the analogous notion of
stratification is from \cite[\S4]{Benson/Iyengar/Krause:2009a}.

\begin{theorem}
\label{th:basechange-cosupport}
Let $(F;\alpha)\col (\sfT;R) \to (\sfU;S)$ be an exact functor which
preserves set-indexed coproducts and products, and fix an object $X$
in $\sfT$.
\begin{enumerate}
\item If $\sfT$ is costratified by $R$ and the right adjoint of $F$ is
  faithful on objects, \nolinebreak then
\[
\supp_{S}FX=(\alpha^*)^{-1} (\supp_{R}X)\,\cap\,\supp_S\sfU.
\]
\item If $\sfT$ is stratified by $R$ and the left adjoint of $F$ is faithful
on objects, then
\[
\cosupp_{S}FX=(\alpha^*)^{-1} (\cosupp_{R}X)\,\cap\,\supp_S\sfU.
\]
\end{enumerate}
\end{theorem}

\begin{proof}
We prove the statement concerning cosupports; the argument for the one
for supports is exactly analogous.

To begin with, from Corollary~\ref{cor:basechange-support} one gets an
inclusion
\[
\cosupp_{S}FX\subseteq(\alpha^*)^{-1}(\cosupp_{R}X)\,\cap\,\cosupp_S\sfU
\]
Now fix a $\fq\in\supp_S\sfU$ with $\fq\not\in\cosupp_S FX$.  We need
to show that $\fp=\alpha^*(\fq)$ is not in $\cosupp_RX$. 
Let $E$ be a left adjoint of
$F$. Using adjunction, one has 
\[\Hom_\sfT(E\gam_\fq -,X)\cong\Hom_\sfU(-,\lam^\fq FX)=0.\]
There exists some object $U$ in $\sfU$ such that $E\gam_\fq U\ne 0$,
since $\fq\in\supp_S\sfU$ and $E$ is faithful on objects. Moreover,
$E\gam_\fq U$ belongs to $\gam_\fp\sfT$, by
Corollary~\ref{cor:basechange-support}.  Since $R$ stratifies $\sfT$,
the subcategory $\gam_{\fp}\sfT$ contains no non-trivial localizing
subcategories, and hence coincides with $\Loc_{\sfT}(E\gam_\sfq U)$.
Thus \[ 0=\Hom_{\sfT}(\gam_{\fp}-,X)\cong \Hom_{\sfT}(-,\lam^{\fp}X),
\]
and therefore $\fp\not\in\cosupp_R X$.
\end{proof}

\subsection*{Perfect generators and cogenerators}

For any subset $\mcU$ of $\Spec R$ we consider the full subcategories
\[
\sfT_\mcU=\{X\in\sfT\mid\supp_RX\subseteq\mcU\} \quad\text{and} \quad
\sfT^\mcU=\{X\in\sfT\mid\cosupp_RX\subseteq\mcU\}\,.
\]
The notion of a set of perfect cogenerators was recalled in
Section~\ref{se:plocalpcomplete}. The notion of a set of perfect
generators is analogous; see \cite{Krause:2002a}.

\begin{lemma}
\label{le:basechange-generate}
Let $(F;\alpha)\col (\sfT;R) \to (\sfU;S)$ be an exact functor which
preserves set-indexed coproducts and products. Let $\mcU$ be a subset
of $\Spec R$ and set $\widetilde\mcU=(\alpha^*)^{-1}\mcU$. Then the
following statements hold:
\begin{enumerate}
\item 
When $F$ is faithful on objects from $\sfT_\mcU$, its left adjoint
maps any set of perfect generators of $\sfU_{\widetilde\mcU}$ to a set
of perfect generators of $\sfT_\mcU$.
\item 
When $F$ is faithful on objects from $\sfT^\mcU$, its right adjoint
maps any set of perfect cogenerators of $\sfU^{\widetilde\mcU}$ to a
set of perfect cogenerators of $\sfT^\mcU$,
\end{enumerate}
\end{lemma}

\begin{proof}
Let $E$ denote a left adjoint of $F$. It follows from
Corollary~\ref{cor:basechange-support} that $F$ and $E$ restrict to
functors between $\sfT_\mcU$ and $\sfU_{\widetilde\mcU}$. Now use
adjunction to prove (1). The proof of (2) is analogous.
\end{proof}

\section{Tensor triangulated categories}
\label{se:tens}

In this section we discuss special properties of triangulated
categories which hold when they have a tensor structure.

Let $(\sfT,\otimes,\one)$ be a tensor triangulated category as defined
in \cite[\S8]{\bik:2008a}. In particular, $\sfT$ is a compactly
generated triangulated category with a symmetric monoidal structure;
$\otimes$ is its tensor product and $\one$ the unit of the tensor
product. The tensor product is exact in each variable and preserves
coproducts.

By Brown representability there are function objects $\fHom(X,Y)$
satisfying
\begin{equation}
\label{eq:adj}
\Hom_\sfT(X \otimes Z,Y) \cong \Hom_\sfT(Z,\fHom(X,Y)),
\end{equation}
and we write $X^\vee$ for the \emph{Spanier--Whitehead dual}
$\fHom(X,\one)$. Note that the adjunction extends to function objects,
in the sense that there are natural isomorphisms
\[ 
\fHom(X \otimes Z,Y) \cong \fHom(Z,\fHom(X,Y))\,.
\]
This is an easy consequence of Yoneda's lemma.

We shall assume that the tensor unit $\one$ is compact and that all
compact objects $C$ are strongly dualizable in the sense that the
canonical morphism
\[ 
C^\vee \otimes X \to \fHom(C,X)
\]
is an isomorphism for all $X$ in $\sfT$.  We also assume that
$\fHom(-,Y)$ is exact for each object $Y$ in $\sfT$.\footnote{The
  exactness of $\fHom(-,Y)$ was omitted from \cite[\S8]{\bik:2008a},
  since it was not used there, but it is important in
  Lemma~\ref{le:loc-coloc} below.}

\subsection*{Canonical actions}
The symmetric monoidal structure of $\sfT$ ensures that the
endomorphism ring $\End^{*}_{\sfT}(\one)$ is graded commutative. It
acts on $\sfT$ via homomorphisms
\[
\End^{*}_{\sfT}(\one)\xra{\ X\otimes-\ }\End^{*}_{\sfT}(X).
\]
In particular, any homomorphism $R\to \End^{*}_{\sfT}(\one)$ of rings
with $R$ graded commutative induces an action of $R$ on $\sfT$. We say
that an $R$ action on $\sfT$ is \emph{canonical} if it arises from
such a homomorphism. In that case there are for each specialization
closed subset $\mcV$ and point $\fp$ of $\Spec R$ natural isomorphisms
\begin{equation}
\label{eq:loc-tensor}
\gam_{\mcV}X\cong X\otimes \gam_{\mcV}\one,\quad \bloc_{\mcV}X\cong
X\otimes \bloc_{\mcV}\one, \quad\text{and}\quad \gam_{\fp}X\cong
X\otimes \gam_{\fp}\one.
\end{equation}
These isomorphisms are from \cite[Theorem~8.2,
  Corollary~8.3]{\bik:2008a}.\footnote{For these results to hold, the
  $R$ action should be canonical, for the $R$-linearity of the
  adjunction isomorphism \eqref{eq:adj} is used in the arguments.}

\begin{proposition}
\label{pr:natisos}
Let $\mcV\subseteq\Spec R$ be a specialization closed subset and
$\fp\in\Spec R$. Given objects $X$ and $Y$ in $\sfT$, there are
natural isomorphisms
\begin{gather*} 
\fHom(\gam_\mcV X,Y)\cong \fHom(X,\lam^\mcV Y) \cong
\lam^\mcV\fHom(X.Y), \\ \fHom(L_\mcV X,Y) \cong \fHom(X,V^\mcV Y)\cong
V^\mcV\fHom(X,Y),\\ \fHom(\gam_\fp X,Y) \cong \fHom(X,\lam^\fp Y)
\cong \lam^\fp\fHom(X,Y).
\end{gather*}
In particular, there are natural isomorphisms
\[ 
\lam^\mcV X \cong \fHom(\gam_\mcV\one,X),\quad V^\mcV X \cong
\fHom(L_\mcV\one,X),\quad \lam^\fp X \cong \fHom(\gam_\fp\one,X)\,.
\] 
\end{proposition}

\begin{proof}
Combine the isomorphisms in \eqref{eq:loc-tensor} with the adjunction
defining $\fHom$.
\end{proof}

\subsection*{Colocalizing subcategories} 
Function objects turn localizing subcategories into colocalizing
subcategories in the following sense.

\begin{lemma}
\label{le:loc-coloc}
Let $\sfC$ be a class of objects in $\sfT$ and $X,Y\in\sfT$.  If $X$
belongs to $\Loc(\sfC)$, then $\fHom(X,Y)$ belongs to
$\Coloc(\{\fHom(C,Y)\mid C\in\sfC\})$.
\end{lemma}

\begin{proof}
This holds as $\fHom(-,Y)$ is exact and turns coproducts into
products.
\end{proof}

We focus attention on colocalizing subcategories satisfying the
equivalent conditions of the following lemma.

\begin{lemma}
\label{le:Hom closed}
Let $\sfS$ be a colocalizing subcategory of $\sfT$. Then the following
conditions on $\sfS$ are equivalent:
\begin{enumerate}
\item 
For all compact objects $X$ in $\sfT$ and all $Y$ in $\sfS$, $X\otimes
Y$ is also in $\sfS$.
\item 
For all compact objects $X$ in $\sfT$ and all $Y$ in $\sfS$,
$\fHom(X,Y)$ is in $\sfS$.
\item 
For all objects $X$ in $\sfT$ and all $Y$ in $\sfS$, $\fHom(X,Y)$ is
in $\sfS$.
\end{enumerate}
\end{lemma}

\begin{proof}
The equivalence of (1) and (2) follows from the isomorphisms
$X^\vee\otimes Y \cong \fHom(X,Y)$ and $X^{\vee\vee}\cong X$. The
equivalence of (2) and (3) follows from Lemma~\ref{le:loc-coloc} and
the fact that $\sfT$ is compactly generated.
\end{proof}

We say that a colocalizing subcategory is \emph{Hom closed} if the
equivalent conditions of the lemma hold, and write
$\Coloc^{\fHom}(\sfC)$ for the smallest Hom closed colocalizing
subcategory containing a class $\sfC$ of objects in $\sfT$.

A localizing subcategory $\sfS$ of $\sfT$ is \emph{tensor ideal} or
\emph{tensor closed} if $X\in\sfT$ and $Y\in\sfS$ imply that $X\otimes
Y$ is in $\sfS$. Given a class $\sfC$ of objects in $\sfT$, we denote
by $\Loc^\otimes(\sfC)$ the smallest tensor ideal localizing
subcategory containing $\sfC$.

The gist of the next result is that (an appropriate version of) the
local-global principle holds for tensor triangulated categories. The
first part, about localizing subcategories, is from \cite[Theorem
  3.6]{\bik:2008b}.

\begin{theorem}
\label{th:tensor-locglob}
For each $X\in\sfT$, there are equalities
\begin{align*}
\Loc_\sfT^\otimes(X)& =\Loc_\sfT^\otimes(\{\gam_\fp X \mid \fp\in\Spec
R\}) \\ \Coloc_\sfT^{\fHom}(X) & =\Coloc_\sfT^{\fHom}(\{\lam^\fp X
\mid \fp\in\Spec R\})\,.
\end{align*}
\end{theorem}

\begin{proof}
The first equality is \cite[Theorem 3.6]{\bik:2008b}.

In particular, $\Loc_\sfT^\otimes(\one)=\Loc_\sfT^\otimes(\{\gam_\fp
\one\, |\, \fp\in\Spec R\})$, which in conjunction with
Lemma~\ref{le:loc-coloc} gives the second equality below:
\begin{align*}
\Coloc_\sfT^{\fHom}(X)&=\Coloc_\sfT^{\fHom}(\fHom(\one,X))\\ &=\Coloc_\sfT^{\fHom}(\{\fHom(\gam_\fp\one,
X) \mid \fp\in\Spec R\})\\ &=\Coloc_\sfT^{\fHom}(\{\lam^\fp X \mid
\fp\in\Spec R\})
\end{align*}
The first one holds because $X=\fHom(\one,X)$ and last one is by
Proposition~\ref{pr:natisos}.
\end{proof}

\begin{lemma}
\label{le:Hom-closed}
If $\sfS$ is a colocalizing subcategory of $\sfT$ then $\fHom(X,Y)$ is
in $\sfS$ for all $X$ in $\Loc_\sfT(\one)$ and $Y$ in $\sfS$. In
particular, if $\one$ generates $\sfT$ then all its colocalizing
subcategories are Hom closed.
\end{lemma}

\begin{proof}
This follows from Lemma~\ref{le:loc-coloc}.
\end{proof}

\begin{remark}
\label{rem:local-global}
Let $\sfT$ be a tensor triangulated category with a canonical
$R$-action. If $\sfT$ is generated by its unit $\one$, then the local
global principle for (co)localizing subcategories holds.  This follows
from Theorem~\ref{th:tensor-locglob} and Lemma~\ref{le:Hom-closed}.
\end{remark}

\section{Costratification}
\label{se:costratification}
Let $\sfT=(\sfT,\otimes,\one)$ be a tensor triangulated category as in
Section~\ref{se:tens}, endowed with a canonical $R$-action. In this
section, we introduce a variant of the notion of costratification, see
Remark~\ref{rem:costratification}, suitable for this context and
explain some consequences, including a classification of the Hom
closed colocalizing subcategories.

Recall from Proposition~\ref{prop:classification} that there are maps
$\sigma$ and $\tau$ which yield a classification of colocalizing
subcategories. Proposition~\ref{pr:natisos} implies that each
$\lam^{\fp}\sfT$ is Hom closed, so these maps restrict to the
following maps on Hom closed subcategories:
\[ 
\left\{\begin{gathered} \text{Hom closed
  colocalizing}\\ \text{subcategories of $\sfT$}
\end{gathered}\;
\right\} \xymatrix@C=3pc{ \ar@<1ex>[r]^-{\sigma} &
  \ar@<1ex>[l]^-{\tau}} \left\{
\begin{gathered}
  \text{families $(\sfS(\fp))_{\fp\in\Spec R}$
    with}\\ \text{$\sfS(\fp)\subseteq \lam^{\fp}\sfT$ a Hom
    closed}\\ \text{colocalizing subcategory}
\end{gathered}\;
\right\}
\] 
where $\sigma(\sfS)=(\sfS\cap\lam^\fp\sfT)$ and $\tau(\sfS(\fp))$ is
the colocalizing subcategory of $\sfT$ cogenerated by all the
$\sfS(\fp)$. The following result is the analogue of
\cite[Proposition~3.6]{\bik:2009a}.

\begin{proposition}
\label{pr:lg=reduction}
The maps $\sigma$ and $\tau$ are mutually inverse bijections.
\end{proposition}

\begin{proof}
The proof is exactly analogous to that of
Proposition~\ref{prop:classification}, using the local-global
principle from Theorem~\ref{th:tensor-locglob}.
\end{proof}

We say that \emph{the tensor triangulated category $\sfT$ is
  costratified by} $R$ if for each $\fp\in\Spec R$, the colocalizing
subcategory $\lam^\fp\sfT$ contains no proper non-zero Hom closed
colocalizing subcategories. Compare this definition with the one in
Remark~\ref{rem:costratification} for general triangulated
categories. One does not have to impose the analogue of the
local-global principle (C1), for it always holds; see
Theorem~\ref{th:tensor-locglob}.

If $\sfT$ is costratified by $R$ then for each $\fp\in\Spec R$ there
are only two possibilities for $\sfS(\fp)$, namely
$\sfS(\fp)=\lam^\fp\sfT$ or $\sfS(\fp)=0$. So the maps $\sigma$ and
$\tau$ reduce to
\[ 
\left\{\begin{gathered} \text{Hom closed
  colocalizing}\\ \text{subcategories of $\sfT$}
\end{gathered}\;
\right\} \xymatrix@C=3pc{ \ar@<1ex>[r]^-{\cosupp_R} &
  \ar@<1ex>[l]^-{\cosupp_R^{-1}}} \left\{
\begin{gathered}
  \text{subsets of $\supp_R\sfT$}
\end{gathered}\;
\right\}.
\] 

The next result is now immediate from the definition of
costratification.

\begin{corollary}
\label{co:coloc-classify}
\pushQED{\qed} If the tensor triangulated category $\sfT$ is
costratified by $R$ then the above maps $\cosupp_R$ and
$\cosupp_R^{-1}$ are mutually inverse bijections.\qedhere
\end{corollary}

\subsection*{Stratification}
In analogy with the notion of costratification, \emph{the tensor
  triangulated category $\sfT$ is stratified by} $R$ if for each
$\fp\in\Spec R$ the localizing subcategory $\gam_\fp\sfT$ is zero or
minimal among tensor ideal localizing subcategories; see
\cite[\S7]{\bik:2009a}.

Next we establish a formula relating support and cosupport
when $\sfT$ is stratified.

\begin{lemma}
\label{le:cosupp-hom}
An inclusion $\cosupp_R\fHom(X,Y)\subseteq\supp_R X\cap \cosupp_R Y$
holds for all objects $X$ and $Y$ in $\sfT$.
\end{lemma}

\begin{proof}
Fix $\fp\in\Spec R$; recall that $\fp$ is in $\cosupp_{R}X$ if and
only if $\fHom(\gam_{\fp}\one,X)\ne 0$. Using \eqref{eq:adj} and
\eqref{eq:loc-tensor}, and Proposition~\ref{pr:natisos}, one gets
isomorphisms
\begin{equation}
\label{eq:cosupp-hom}
\fHom(\gam_\fp\one,\fHom(X,Y))\cong\fHom(\gam_\fp X,Y)\cong
\fHom(X,\lam^{\fp}Y)\,.
\end{equation}
It follows that $\gam_{\fp}X\ne 0$ and $\lam^{\fp}Y\ne 0$ when $\lam^{\fp}\fHom(X,Y)\ne 0$.
\end{proof}

In \cite[Theorem~7.3]{\bik:2009a} we proved that $\supp_R(X\otimes
Y)=\supp_{R}X\cap \supp_{R}Y$ holds if $\sfT$ is stratified. An
analogue for function objects is contained in the next result.

\begin{theorem}
\label{th:cosupp-hom}
The following conditions on $\sfT$ are equivalent:
\begin{enumerate}
\item The tensor triangulated category $\sfT$ is stratified by $R$. 
\item $\cosupp_R\fHom(X,Y)=\supp_R X\cap \cosupp_R Y$ for all $X,Y\in\sfT$.
\item $\fHom(X,Y)=0$ implies $\supp_R X\cap \cosupp_R Y=\varnothing$ for all $X,Y\in\sfT$.
\end{enumerate}
\end{theorem}

\begin{proof}
(1) $\Rightarrow$ (2): One inclusion follows from
  Lemma~\ref{le:cosupp-hom}.  For the other inclusion one uses that
  $\sfT$ is stratified by $R$.  Fix $\fp\in\supp_R X\cap\cosupp_R Y$.
  The minimality of the tensor ideal localizing subcategory
  $\gam_{\fp}\sfT$ implies $\gam_\fp\one\in\Loc^\otimes(\gam_{\fp}X)$,
  since $\gam_{\fp}X\ne 0$. Applying Lemma~\ref{le:loc-coloc}, one obtains
\[ 
0 \ne\fHom(\gam_\fp\one,Y)\in\Coloc^{\fHom}(\fHom(\gam_{\fp}X,Y))
\]
and therefore $\fHom(\gam_{\fp}X, Y)\ne 0$. Using the first isomorphism in
\eqref{eq:cosupp-hom}, it follows that $\fp$ is in the cosupport of
$\fHom(X,Y)$.

(2) $\Rightarrow$ (3): Clearly, $\fHom(X,Y)= 0$ implies
$\cosupp_R\fHom(X,Y)=\varnothing$.

(3) $\Rightarrow$ (1): To prove that the tensor triangulated category
$\sfT$ is stratified by $R$, it suffices to show that given non-zero
objects $X$ and $Y$ in $\gam_{\fp}T$ for some $\fp$ in $\Spec R$,
there exists a $Z$ such that $\Hom^{*}_{\sfT}(X\otimes Z,Y)\ne 0$; see
\cite[Lemma~3.9]{\bik:2008a}.

Since $\supp_{R}Y=\{\fp\}$, it follows from Theorem~\ref{th:max} that
$\fp\in\cosupp_{R}Y$ holds, and hence from our assumption that
$\fHom(X,Y)\ne 0$. In particular, there exists a $Z\in\sfT$ such that
$\Hom_{\sfT}^{*}(Z,\fHom(X,Y))\ne 0$. The adjunction isomorphism
\eqref{eq:adj} then yields $\Hom_{\sfT}^{*}(X\otimes Z,Y)\ne 0$.
\end{proof}

The preceding result has the following immediate consequence.

\begin{corollary}
\label{co:hom-vanishing}
\pushQED{\qed} Suppose the tensor triangulated category $\sfT$ is generated
by its unit. Then $\sfT$ is stratified by $R$ if and only if for all
objects $X$ and $Y$ in $\sfT$ one has
\[
\Hom_\sfT^*(X,Y)=0\quad\iff\quad\supp_RX\cap\cosupp_RY=\varnothing.\qedhere
\]
\end{corollary}

There is the following connection between stratification and
costratification.

\begin{theorem}
\label{th:strat-costrat}
When the tensor triangulated category $\sfT$ is costratified by $R$,
it is also stratified by $R$, and then there is an equality
\[ 
\cosupp_R\fHom(X,Y)=\supp_R X\cap \cosupp_R Y\quad\text{for all}\quad X,Y\in\sfT.
\]
\end{theorem}

\begin{proof}
It suffices to prove that given non-zero objects $X$ and $Y$ in
$\gam_{\fp}T$ for some $\fp$ in $\Spec R$, there exists a $Z$ such
that $\Hom^{*}_{\sfT}(X\otimes Z,Y)\ne 0$; see
\cite[Lemma~3.9]{\bik:2008a}.

Assume $\sfT$ is costratified by $R$. As $\gam_{\fp}X\ne 0$ there
exist an object $C\in\sfT$ such that $\Hom_{\sfT}^{*}(X,C(\fp))\ne 0$,
by Proposition~\ref{pr:hom-supp}. It is easy to verify using the
adjunction isomorphism~\eqref{eq:adj} that the subcategory
\[
\sfS=\{W\in\sfT\mid \text{$\Hom_{\sfT}^{*}(X\otimes Z,W)=0$ for all
  $Z\in\sfT$}\}
\]
of $\sfT$ is colocalizing and Hom closed. 

Now observe that $\fp\in\cosupp_{R}Y$ holds by Theorem~\ref{th:max},
since $\supp_{R}Y=\{\fp\}$. Thus
$\lam^{\fp}\sfT=\Coloc^{\fHom}(\lam^{\fp}Y)$, by the costratification
hypothesis. This implies $\lam^\fp Y\not\in\sfS$, since
$C(\fp)\in\lam^{\fp}\sfT$ by Theorem~\ref{th:max}, and
$C(\fp)\not\in\sfS$.  Thus one obtains that \[\Hom_{\sfT}^{*}(X\otimes
Z,Y)\cong\Hom_{\sfT}^{*}(\gam_\fp (X\otimes
Z),Y)\cong\Hom_{\sfT}^{*}(X\otimes Z,\lam^{\fp}Y)\ne 0\] for some $Z$
in $\sfT$, and hence that $\sfT$ is stratified.

The formula for $\cosupp_R\fHom(X,Y)$ follows from
Theorem~\ref{th:cosupp-hom}.
\end{proof}

\begin{remark}
It is an open question whether stratification implies
costratification. The proof of Theorem~\ref{th:strat-costrat} uses the
fact that every localizing subcategory generated by a set of objects
arises as the kernel of a localization functor; see the proof of
\cite[Lemma~3.9]{\bik:2008a}. It is not known whether the analogous
statement for colocalizing subcategories is true or not. This reflects
the fact that products are usually more complicated than coproducts.
\end{remark}

The following corollary combines the classification of colocalizing
subcategories, Corollary~\ref{co:coloc-classify}, with the
classification of localizing subcategories in
\cite[Theorem~3.8]{\bik:2008b}.

\begin{corollary}
\label{co:loccoloc}
If the tensor triangulated category $\sfT$ is costratified by $R$ then
the map sending a subcategory $\sfS$ to $\sfS^\perp$ induces a
bijection
\[ 
\left\{\begin{gathered} \text{tensor closed
  localizing}\\ \text{subcategories of $\sfT$}
\end{gathered}\;
\right\} \xymatrix@C=3pc{ \ar[r]^-{\sim} &} \left\{
\begin{gathered}
  \text{Hom closed colocalizing}\\ \text{subcategories of $\sfT$}
\end{gathered}\;
\right\}\,.
\]
The inverse map sends a Hom closed colocalizing subcategory $\sfU$ to $^\perp\sfU$. 
\end{corollary}

\begin{proof}
Assume $\sfT$ is costratified by $R$; it is then stratified by $R$, by
Theorem~\ref{th:strat-costrat}.  In particular, both the tensor closed
localizing subcategories and the Hom closed colocalizing subcategories
of $\sfT$ are in bijection with the subsets of $\supp_{R}\sfT$, via
the maps $\supp_{R}(-)$ and $\cosupp_{R}(-)$, respectively; see
\cite[Theorem~3.8]{\bik:2008a} and
Corollary~\ref{co:coloc-classify}. Now, for any tensor closed
localizing subcategory $\sfS$ of $\sfT$ one has equalities
\begin{align*}
\sfS^{\perp} &=\{Y\in\sfT\mid \text{$\fHom(X,Y)=0$ for all
  $X\in\sfS$}\} \\ &=\{Y\in\sfT\mid
\cosupp_{R}Y\cap\supp_{R}\sfS=\varnothing\}
\\ &=\cosupp_{R}^{-1}(\supp_{R}\sfT\setminus \supp_{R}\sfS)
\intertext{where the first one is a routine verification, and the
  second one is by Theorem~\ref{th:strat-costrat}. In the same vein,
  for any Hom closed localizing subcategory $\sfU$ one has}
          {}^{\perp}\sfU&=\supp_{R}^{-1}(\supp_{R}\sfT\setminus
          \cosupp_{R}\sfU).
\end{align*}
It thus follows that, under the identification above, both maps
$\sfS\mapsto \sfS^{\perp}$ and $\sfU\mapsto {}^{\perp}\sfU$ correspond
to the map on $\supp_{R}\sfT$ sending a subset to its complement, and
are thus inverse to each other.
\end{proof}

\subsection*{Brown--Comenetz duality}
Let $k$ be a commutative ring and suppose that the category $\sfT$ is
$k$-linear. We denote by $D=\Hom_k(-,E)$ the duality for the category
of $k$-modules with respect to a fixed injective cogenerator $E$.

The \emph{Brown--Comenetz dual} $X^*$ of an object $X$ is defined by
the isomorphism
\[
D\Hom_\sfT(\one,-\otimes X)\cong\Hom_\sfT(-,X^*).
\]
Note that there is a natural isomorphism $X^*\cong\fHom(X,\one^*)$.

\begin{proposition}
\label{pr:cosupp-one}
One has $\cosupp_R\one^*=\supp_R\sfT$ and $\cosupp_R X^*\subseteq
\supp_R X$ for any $X\in\sfT$; equality holds if $\sfT$ is
stratified by $R$ as a tensor triangulated category.
\end{proposition}

\begin{proof}
The first equality holds because $X^*\cong\fHom(X,\one^*)$ and $X^*=0$
if and only if $X=0$. The inclusion follows from
Lemma~\ref{le:cosupp-hom}, since
\[
\cosupp_R X^*=\cosupp_R \fHom(X,\one^*)\subseteq\supp_R
X\cap\cosupp_R\one^*=\supp_R X.
\]
When $\sfT$ is stratified, Theorem~\ref{th:cosupp-hom} gives equality.
\end{proof}

\begin{remark}
\label{rem:costratification-tt}
Let $\sfT$ be a tensor triangulated category generated by its unit
$\one$, and equipped with a canonical $R$ action. Given
Remark~\ref{rem:local-global}, it is clear that $\sfT$ is
(co)stratified by $R$ as a triangulated category if, and only if, it
is (co)stratified by $R$ as a tensor triangulated category. And to
verify that $\sfT$ is stratified (respectively, costratified) it
suffices to check that each $\gam_{\fp}\sfT$ is a minimal localizing
subcategory (respectively, each $\lam^{\fp}\sfT$ is a minimal
colocalizing subcategory).

The results in this section thus yield interesting information about
all localizing and colocalizing subcategories of $\sfT$.  For
instance, Corollary~\ref{co:coloc-classify} coincides with the
bijection in Remark~\ref{rem:costratification}, while
Theorem~\ref{th:strat-costrat} says that if $\sfT$ is costratified
then it is also stratified.
\end{remark}

\section{Formal dg algebras}
\label{se:formal}
In this section we prove that the derived category of dg (short for
``differential graded'') modules over a formal dg algebra is
costratified by its cohomology algebra, when that algebra is graded
commutative and noetherian.

Let $A$ be a dg algebra and $\sfD(A)$ its derived category of (left)
dg modules. It is a triangulated category, generated by the compact
object $A$; see~\cite{Keller:1994a}. A morphism $A\to B$ of dg
algebras is a \emph{quasi-isomorphism} if the induced map $\hh A\to\hh
B$ is an isomorphism. Then restriction induces an equivalence of
triangulated categories $\sfD(B)\xra{\sim}\sfD(A)$, with quasi-inverse
the functor $B\lotimes_A-$. Dg algebras $A$ and $B$ are called
\emph{quasi-isomorphic} if there is a finite chain of
quasi-isomorphisms linking them.

The multiplication on $A$ induces one on its cohomology $\hh A$. We
say $A$ is \emph{formal} if it is quasi-isomorphic to $\hh A$, viewed
as a dg algebra with zero differential.

\begin{remark}
\label{re:cdga}
Suppose that $A$ is formal and that $\hh A$ is graded commutative. Fix
a chain of quasi-isomorphisms linking $A$ and $\hh A$; it induces an
equivalence of categories $\sfD(A)\simeq \sfD(\hh A)$.

The derived tensor product of dg modules endows $\sfD(\hh A)$ with a
structure of a tensor triangulated category, with unit $\hh A$. Thus,
$\sfD(A)$ also acquires such a structure, via the equivalence
$\sfD(A)\simeq \sfD(\hh A)$; denote $\otimes$ this tensor product on
$\sfD(A)$. The object $A$ is a tensor unit, and one gets an action of
$\hh A$ on $\sfD(A)$, defined by taking for each object $X$ the
composite map
\[
\hh
A\xra{\sim}\End^*_{\sfD(A)}(A)\xra{\ X\otimes-\ }\End^{*}_{\sfD(A)}(X)\,.
\]
We refer to this as the action induced by the given chain of
quasi-isomorphisms linking $A$ and $\hh A$.
\end{remark}

The action of $\hh A$ on $\sfD(A)$ depends on the choice of
quasi-isomorphisms linking $A$ and $\hh A$. One has however the
following independence statement.

\begin{lemma}
\label{le:cdga-fun}
Let $A$ be a formal dg algebra with $\hh A$ graded commutative and
noetherian. For any specialization closed set $\mcV\subseteq\Spec\hh
A$ the functors $\gam_\mcV$, $L_\mcV$, $\lam^\mcV$ and $V^\mcV$ are
independent of a chain of quasi-isomorphisms linking $A$ to $\hh A$.
\end{lemma}

\begin{proof}
For any chain of quasi-isomorphisms, it is clear from the construction
that the homomorphism $\hh A\to \End^*_{\sfD(A)}(A)$ is the canonical
one, and hence it is independent of the action. It thus remains to
note that the local cohomology functors on a compactly generated
$R$-linear triangulated category are determined by the action of $R$
on a compact generator, by \cite[Corollary~3.3]{\bik:2008b}.
\end{proof}

In the case when $A$ is a ring, which may be viewed as a dg algebra
concentrated in degree zero, the result below is contained in recent
work of Neeman~\cite{Neeman:coloc}. We note that the cosupport, and
hence the costratification, are independent of a choice of a canonical
action, by Lemma~\ref{le:cdga-fun}.

\begin{theorem}
\label{th:cdga}
Let $A$ be a formal dg algebra such that $\hh A$ is graded commutative
and noetherian.  The category $\sfD(A)$ is costratified by any $\hh
A$-action induced by a chain of quasi-isomorphisms linking $A$ and
$\hh A$.
\end{theorem}

\begin{proof}
We may replace $A$ by $\hh A$ and assume $d^{A}=0$ and $A$ is graded
commutative. Set $\sfD=\sfD(A)$. Since $A$ is a unit and a generator
of this tensor triangulated category, its colocalizing subcategories
are $\RHom$ closed; see Lemma~\ref{le:Hom-closed}.  It remains to
verify that $\lam^{\fp}\sfD$ is a minimal colocalizing subcategory for
each homogeneous prime ideal $\fp$ in $A$; see
Remark~\ref{rem:costratification-tt}.

Let $k(\fp)$ be the graded residue field at $\fp$. The object $k(\fp)$
is in $\lam^\fp\sfD$, so to verify the minimality of $\lam^{\fp}\sfD$
it suffices to note that for each non-zero $M$ in $\lam^\fp\sfD$, the
following equalities hold:
\begin{align*}
\Coloc_{\sfD}(M) &= \Coloc_{\sfD}(\RHom_A(\gam_\fp A,M))\\ &
=\Coloc_{\sfD}(\RHom_A(k(\fp),M))\\ & =\Coloc_{\sfD}(k(\fp))
\end{align*}
The first equality holds by Proposition~\ref{pr:natisos} since $M\cong
\lam^{\fp}M$. As to the second one, since $A$ is stratified by $\hh
A$, by \cite[Theorem~8.1]{\bik:2009a}, one has
\[
\Loc_\sfD(\gam_\fp A)=\gam_\fp\sfD=\Loc_\sfD(k(\fp))\,.
\] 
Now apply Lemma~\ref{le:loc-coloc}. The last equality follows from the
fact that $k(\fp)$ is a graded field and the action of $A$ on
$\RHom_A(k(\fp),M)$ factors through $k(\fp)$.
\end{proof}

\subsection*{Exterior algebras}
Let $\Lambda$ be a graded exterior algebra over a field $k$ on
indeterminates $\xi_1,\dots,\xi_c$ in negative odd degree, regarded as
a dg algebra with zero differential. We give $\Lambda$ a structure of
Hopf algebra via $\Delta(\xi_i)=\xi_i\otimes 1 + 1 \otimes \xi_i$.  In
\cite[\S4]{\bik:2008b}, we introduced the homotopy category of
graded-injective dg modules over $\Lambda$, denoted $\KInj{\Lambda}$,
and proved the following statements: $\KInj{\Lambda}$ is a compactly
generated tensor triangulated category, in the sense of
Section~\ref{se:tens}; its unit is an injective resolution of the
trivial module $k$, and this generates $\KInj{\Lambda}$; the graded
endomorphism algebra of the unit is $\Ext^*_\Lambda(k,k)$, which is
isomorphic to the graded polynomial $k$-algebra $S=k[x_1,\dots,x_c]$
with $|x_i|=-|\xi_i|+1$.

\begin{theorem}
\label{th:ext-costrat}
The category $\KInj{\Lambda}$ is costratified by the action of
$\Ext^*_\Lambda(k,k)$.
\end{theorem}

\begin{proof}
It is proved in \cite[Theorem~6.2]{\bik:2008b} that a suitable dg
module $J$ over $\Lambda \otimes_k S$ yields an equivalence
$\Hom_\Lambda(J,-)\col\KInj{\Lambda}\xra{\sim} D(S)$. The theorem now
follows from Theorem~\ref{th:cdga}.
\end{proof}

\section{Finite groups}
\label{se:finitegroups}
Throughout this section, $G$ will be a finite group and $k$ a field
whose characteristic divides the order of $G$.  The associated group
algebra is denoted $kG$, and $H^*(G,k)$ denotes its cohomology
$k$-algebra, $\Ext^{*}_{kG}(k,k)$. This algebra is graded commutative,
because $kG$ is a Hopf algebra, and finitely generated and hence
noetherian, by a result of Evens and Venkov
\cite{Evens:1961a,Venkov:1959a}; see also Golod \cite{Golod:1959a}.

Let $\KInj{kG}$ be the homotopy category of complexes of injective
$kG$-modules. This is a compactly generated tensor triangulated
category, in the sense of Section~\ref{se:tens}, where the tensor
product $X\otimes _k Y$ and the function object $\Hom_k(X,Y)$ are
induced by those on $kG$-modules via the diagonal action of $G$. The
injective resolution $\sfi k$ of the trivial module $k$ is the
identity for the tensor product, and it yields a canonical action of
$H^*(G,k)$ on $\KInj{kG}$; see \cite{Benson/Krause:2008a} for
details. To simplify notation, we write $\mcV_G$ for $\Spec H^*(G,k)$,
and set for each object $X$ in $\KInj{kG}$
\[
\supp_G X =\supp_{H^{*}(G,k)}X\quad\text{and}\quad \cosupp_G X
=\cosupp_{H^{*}(G,k)}X.
\]

\begin{example}
\label{ex:cosuppkg}
We write $\fm$ for the maximal ideal $H^{\ges 1}(G,k)$.
\begin{enumerate}
\item
One has $\cosupp_{G}X = \{\fm\}$ for any non-zero compact object $X\in\KInj{kG}$.
\item
For each $\fp\in\mcV_{G}$ the object $T_{\sfi k}(I(\fp))$, from \eqref{eq:TI}, satisfies
\[
\supp_{G}T_{\sfi k}(I(\fp)) =\{\fp\} \quad\text{and}\quad \cosupp_G
T_{\sfi k}(I(\fp))=\{\fq\in\mcV_{G} \mid\fq\subseteq\fp\}\,.
\]
\end{enumerate}

Indeed, for (1) observe that the natural morphism $X\to X^{**}$ is an
isomorphism for each compact object $X$ by
\cite[Lemma~11.5]{\bik:2008b}, where $(-)^{*}$ denotes Brown--Comenetz
duality.  Proposition~\ref{pr:cosupp-one} therefore gives the first
inclusion below
\[
\cosupp_{G} X \subseteq \supp_{G}X^{*} \subseteq \{\fm\}\,.
\]
For the second inclusion  one uses that the $H^*(G,k)$-module
$\Hom^*_{\sfK(kG)}(C,X^*)$ is $\fm$-torsion for each compact
$C\in\KInj{kG}$, since it is of the form $\Hom_k(M,k)$ for some
finitely generated $H^*(G,k)$-module $M$ by the defining isomorphism
of the Brown--Comenetz dual $X^*$. It remains to observe that
$\cosupp_{G} X \neq\varnothing$ since $X\neq 0$.

To prove (2), note that $\Hom^*_{\sfK(kG)}(C,T_{\sfi k}(I(\fp)))$ is
$\fp$-torsion and $\fp$-local for each compact object $C\in\KInj{kG}$, by the
isomorphism \eqref{eq:TI} defining $T_{\sfi k}(I(\fp))$.  Thus
$\supp_{G}T_{\sfi k}(I(\fp)) =\{\fp\}$, while the cosupport is given
by Proposition~\ref{pr:perf-cogen}.
 \end{example}

\subsection*{Restriction and induction}
For each subgroup $H$ of $G$ restriction and induction yield exact
functors
\begin{align*}
(-)\da_{H}&=\Hom_{kG}(kG,-)\col \KInj{kG}\to \KInj{kH}\quad\text{and}\quad
  \\ (-)\ua^{G} &= -\otimes_{kH}kG\col \KInj{kH}\to \KInj{kH}\,.
\end{align*}
Restriction yields also a homomorphism $\res_{G,H}\col H^*(G,k) \to
H^*(H,k)$ of graded rings, and hence a map:
\[ 
\res_{G,H}^*\col \mcV_H \to \mcV_G\,.
\] 
It is easy to verify that these functors fit in the framework of
Section~\ref{se:change of categories}:

\begin{lemma}
\label{le:resin-change}
The functor $(-)\da_{H}$ is $\res_{G,H}$-linear and induces an exact
functor
\[
((-)\da_{H},\res_{G,H})\col (\KInj{kG},H^{*}(G,k))\to
(\KInj{kH},H^{*}(H,k))
\]
which preserves coproducts and products. Its right adjoint and left
adjoint is $(-)\ua^{G}$, and the latter is faithful on objects. \qed
\end{lemma}

The following result is an analogue of \cite[Lemma~9.3]{\bik:2008b}.

\begin{lemma}
\label{le:homol-res-ind}
Let $H$ be a subgroup of $G$. Fix $\fp\in\mcV_G$ and set
$\mcU=(\res_{G,H}^*)^{-1}\{\fp\}$. The set $\mcU$ is finite and
discrete, and for any $X\in\KInj{kG}$ and any $Y\in\KInj{kH}$ there
are natural isomorphisms:
\begin{gather*} 
(\gam_\fp X)\da_H\cong\coprod_{\fq\in\mcU}\gam_\fq(X\da_H) \qquad
  \gam_\fp(Y\ua^G)\cong\coprod_{\fq\in\mcU}(\gam_\fq Y)\ua^G
  \\ (\lam^\fp X)\da_H\cong\prod_{\fq\in\mcU}\lam^\fq(X\da_H) \qquad
  \lam^\fp(Y\ua^G)\cong\prod_{\fq\in\mcU}(\lam^\fq Y)\ua^G
\end{gather*}
\end{lemma}

\begin{proof}
The set $\mcU$ is finite and discrete as $H^{*}(H,k)$ is finitely
generated as a module over $H^{*}(G,k)$.  The result now follows from
Corollary~\ref{cor:basechange-fibre-gen}, given
Lemma~\ref{le:resin-change}.
\end{proof}

The next result is an analogue of \cite[Proposition~9.4]{\bik:2008b}.

\begin{proposition}
\label{pr:V-ind}
Let $H$ be a subgroup of $G$. The following statements hold:
\begin{enumerate}
\item
For any object $Y$ in $\KInj{kH}$, there is an equality
\[
\cosupp_G(Y\ua^G)=\res_{G,H}^*(\cosupp_HY)\,.
\]
\item
Any object $X$ in $\KInj{kG}$ satisfies $X\da_H\ua^G\in\Coloc^{\fHom}
(X)$, and hence
\[
\cosupp_G(X\da_H\ua^G)\subseteq\cosupp_GX\,.
\]
\end{enumerate}
\end{proposition}

\begin{proof}
Part (1) is a special case of Corollary~\ref{cor:basechange-support},
which applies by Lemma~\ref{le:resin-change}.

(2) One has $X\da_H\ua^G\cong X \otimes_k k(G/H)$, where $k(G/H)$ is
the permutation module on the cosets of $H$ in $G$. Let $W$ be an
injective resolution over $kG$ of $k(G/H)$. Arguing as in the proof of
\cite[Proposition~5.3]{Benson/Krause:2008a}, one can prove that the
natural map $X \otimes_k k(G/H)\to X \otimes_k W$ is an isomorphism in
$\KInj{kG}$. As $W$ is compact in $\KInj{kG}$, see
\cite[Lemma~2.1]{Krause:2005a}, it remains to apply Lemma~\ref{le:Hom
  closed}.
\end{proof}

\subsection*{Elementary abelian groups}

Let $E$ be an elementary abelian $p$-group
\[ 
E = \langle g_1,\dots,g_r\rangle
\]
and $k$ a field of characteristic $p$. We set $z_i=g_i-1$, so that the
group algebra may be described as
\[ 
kE = k[z_1,\dots,z_r]/(z_1^p,\dots,z_r^p)\,.
\]
Let $A$ be the Koszul complex on $z_1,\dots,z_r$, viewed as a dg
algebra: The graded algebra underlying it is generated by $kE$ in
degree zero together with exterior generators $y_1,\dots,y_r$, each of
degree $-1$.  The differential on $A$ is given by $d(y_i)=z_i$ and
$d(z_i)=0$.  We write $\KInj{A}$ for the homotopy category of
graded-injective dg $A$-modules. It is a compactly generated tensor
triangulated category, with a canonical action of $\Ext^{*}_{A}(k,k)$;
see \cite[\S8]{\bik:2008b}.

\begin{proposition}
\label{pr:koszuldga}
The category $\KInj{A}$ is costratified by the action of
$\Ext^*_A(k,k)$.
\end{proposition}

\begin{proof}
Let $\Lambda$ be an exterior algebra over $k$ on indeterminates
$\xi_1,\dots,\xi_r$ of degree $-1$, viewed as a dg algebra with
$d^{\Lambda}=0$. Let $\KInj{\Lambda}$ be the homotopy category of
graded-injective dg $\Lambda$-modules, with tensor triangulated
structure described in the paragraph preceding
Theorem~\ref{th:ext-costrat}.

By \cite[Lemma 7.1]{\bik:2008b} there is a quasi-isomorphism of dg
$k$-algebras $\phi\col\Lambda\to A$ defined by
$\phi(\xi_i)=z_i^{p-1}y_i$, and by \cite[Proposition~4.6]{\bik:2008b}
this induces an equivalence of triangulated categories
\[ 
\Hom_\Lambda(A,-)\col\KInj{\Lambda}\xra{\sim}\KInj{A}\,.
\]
The desired result is now a consequence of
Theorem~\ref{th:ext-costrat}.
\end{proof}

The next result complements \cite[Theorem~8.1]{\bik:2008b}, concerning
stratification.

\begin{theorem}
\label{th:elab}
Let $E$ be an elementary abelian $p$-group and $k$ a field of
characteristic $p$. The category $\KInj{kE}$ is costratified by the
canonical action of $H^{*}(E,k)$.
\end{theorem}

\begin{proof}
Write $kE\cong k[z_{1},\dots,z_{r}]/(z_{1}^{p},\dots,z_{r}^{p})$, and
let $A$ be the Koszul dg algebra described above. Note that $kE=A_{0}$
so the inclusion $kE\to A$ is a morphism of dg algebras; restriction along
it gives an exact functor $\KInj{A}\to \KInj{kE}$ which preserves
coproducts and products, and is compatible with the induced
homomorphism
$\alpha\col\Ext^*_A(k,k)\to\Ext^*_{kE}(k,k)=H^{*}(E,k)$. The functor
$\Hom_{kE}(A,-)$ is a right adjoint of restriction. We claim that each
$X\in\KInj{kE}$ satisfies an equality:
\[
\Thick(\Hom_{kE}(A,X))=\Thick(X)\quad \text{in $\KInj{kE}$},\tag{$\ast$}
\]
and, in particular, that $\Hom_{kE}(A,-)$ is faithful on objects.

Indeed, recall that $A$ is the Koszul complex over $kE$ on $\bsz =
z_{1},\dots,z_{r}$, and in particular, a finite free complex of
$kE$-modules. This yields the first isomorphism below of complexes of
$kE$-modules:
\[
\Hom_{kE}(A,X)\cong \Hom_{kE}(A,kE)\otimes_{kE}X \cong \Si^{r}
A\otimes_{kE}X
\]
The second one is by self-duality of the Koszul complex;
see~\cite[Proposition~1.6.10]{Bruns/Herzog:1998a}. The radical of the
ideal $(\bsz)$ coincides with that of $(0)$, so the complexes $A$ and
$kE$ generate the same thick subcategory in $\KInj{kE}$; see
\cite[Lemma~6.0.9]{Hovey/Palmieri/Strickland:1997a}. Since
$-\otimes_{kE}X$ is an exact functor on $\KInj{kE}$, it follows from
the isomorphism above that $\Hom_{kE}(A,X)$ and $X$ generate the same
thick subcategory, as desired.

We are thus in a position to apply
Corollary~\ref{cor:basechange-support}. Fix a prime $\fp\in\mcV_{E}$
and let $\fq=\alpha^{-1}(\fp)$. The functor $\Hom_{kE}(A,-)$ then takes
$\lam^\fp\KInj{kE}$ to $\lam^\fq\KInj{A}$ and is faithful. Hence, for
any non-zero objects $X$ and $Y$ in $\lam^\fp\KInj{kE}$, one has
\[
\Hom_{kE}(A,Y)\in\Coloc(\Hom_{kE}(A,X))\quad \text{in $\KInj{A}$},
\]
since $\KInj{A}$ is costratified by the action of $\Ext^*_{A}(k,k)$;
see Proposition~\ref{pr:koszuldga}. It implies, in view of the
equality ($\ast$) above, that in $\KInj{kE}$ one has
\[
Y \in \Thick(\Hom_{kE}(A,Y)) \subseteq \Coloc(\Hom_{kE}(A,X))=\Coloc(X).
\]
Thus, $\lam^\fp\KInj{kE}$ is a minimal colocalizing subcategory. Thus,
$\KInj{kE}$ is costratified by the action of $H^{*}(E,k)$; see
Remark~\ref{rem:costratification-tt}.
\end{proof}

Next we prepare for a version of the preceding theorem for arbitrary
finite groups.

\subsection*{Quillen's stratification}
We consider pairs $(H,\fq)$ of subgroups $H$ of $G$ and primes
$\fq\in\mcV_H$, and say that $(H,\fq)$ and $(H',\fq')$ are
\emph{$G$-conjugate} if conjugation with some element in $G$ takes $H$
to $H'$ and $\fq$ to $\fq'$.

\begin{lemma}
\label{le:conj}
Let $(H,\fq)$ and $(H',\fq')$ be $G$-conjugate pairs and $X$ in
$\KInj{kG}$. Then $\lam^\fq(X\da_H)\ne 0$ if and only if
$\lam^{\fq'}(X\da_{H'})\ne 0$.
\end{lemma}

\begin{proof}
Conjugation with any element $g\in G$ induces an automorphism of
$\KInj{kG}$ that takes an object $X$ to $X^g$. Note that
multiplication with $g$ induces an isomorphism $X\xra{\sim}X^g$. The
assertion follows, since conjugation commutes with restriction and
local homology.
\end{proof}

Quillen has proved that for each $\fp$ in $\mcV_G$ there exists a pair
$(E,\fq)$ such that $E$ is an elementary abelian subgroup of $G$ and
$\res_{G,E}^*(\fq)=\fp$; see the discussion after \cite[Proposition
  11.2]{Quillen:1971b}. We say $\fp$ \emph{originates} in such a pair
$(E,\fq)$ if there does not exist another such pair $(E',\fq')$ with
$E'$ a proper subgroup of $E$.  In this language,
\cite[Theorem~10.2]{Quillen:1971b} reads:

\begin{theorem}
\label{th:quillen}
For any $\fp\in\mcV_G$ all pairs $(E,\fq)$ where $\fp$ originates are
$G$-conjugate. This sets up a one to one correspondence between primes
$\fp$ in $\mcV_G$ and $G$-conjugacy classes of such pairs
$(E,\fq)$. \qed
\end{theorem}

The proof of the next result can be shortened considerably by invoking
the subgroup theorem for supports, \cite[Theorem~11.2]{\bik:2008a},
which is deduced from the stratification theorem for $\KInj{kG}$,
\cite[Theorem~9.7]{\bik:2008a}. We give a direct proof, by extracting
an argument from the proof of the latter result.

\begin{proposition}
\label{pr:Quillen}
Fix $\fp\in\mcV_G$ and suppose $\fp$ originates in $E$. Given a
non-zero object $X$ in $\lam^\fp\KInj{kG}$, one has
\[
X\da_E=\prod_{\res_{G,E}^*(\fq)=\fp}\lam^\fq(X\da_E)
\] 
and $\lam^\fq(X\da_E)\neq 0$ for each such $\fq$.
\end{proposition}

\begin{proof}
The decomposition of $X\da_E$ is by Lemma~\ref{le:homol-res-ind}. For
the remaining statements, it suffices to find \emph{one} pair
$(E',\fq')$ where $\fp$ originates and such that
$\lam^{\fq'}(X\da_{E'})\ne 0$. Then Theorem~\ref{th:quillen} yields
that each pair $(E,\fq)$ as in the statement is $G$-conjugate to
$(E',\fq')$, and hence $\lam^\fq (X\da_E)\ne 0$, by
Lemma~\ref{le:conj}.

Since $X\ne 0$ holds, Chouinard's theorem for $\KInj{kG}$, see
\cite[Proposition 9.6(3)]{\bik:2008a}, provides an elementary abelian
subgroup $E''$ of $G$ such that $X\da_{E''}\ne 0$. Fix a $\fq''$ in
$\cosupp_{E''}(X\da_{E''})$. Proposition~\ref{pr:V-ind} then yields
\[
\res^{*}_{G,E''}(\fq'')\in \cosupp_{G}(X\da_{E''}\ua^{G})\subseteq
\cosupp_{G}X=\{\fp\}\,.
\]
Applying Theorem~\ref{th:quillen} to $E''$ one gets a pair
$(E',\fq')$, with $E'$ a subgroup of $E''$, where $\fq''$
originates. Observe that $\fp$ then originates in $(E',\fq')$, by
functoriality of restrictions and the computation above. It remains to
note that
\[
\fq'\in (\res^{*}_{E'',E'})^{-1}(\cosupp_{E''}(X\da_{E''})) =
\cosupp_{E'}(X\da_{E'})\,,
\] 
where the inclusion holds by the choice of $\fq'$ and the equality is
by Theorem~\ref{th:basechange-cosupport}, applied to the functor
\[
((-)\da_{E'},\res_{E'',E'})\col (\KInj{kE''},H^{*}(E'',k))\to
(\KInj{kE'},H^{*}(E',k))\,;
\]
noting that the hypotheses are satisfied by
Lemma~\ref{le:resin-change} and Theorem~\ref{th:elab}.
\end{proof}

\subsection*{Finite groups}
The following theorem is an analogue of
\cite[Theorem~9.7]{\bik:2008b}, which establishes the stratification
of $\KInj{kG}$; it contains Theorem~\ref{th:elab}, but the latter
statement is used in the proof, both directly and by way of
Proposition~\ref{pr:Quillen}.

\begin{theorem}
\label{th:finite-groups}
The tensor triangulated category $\KInj{kG}$ is costratified by the
canonical action of the cohomology algebra $H^*(G,k)$.
\end{theorem}

\begin{proof}
Fix $\fp\in\mcV_G$ and suppose it originates in $E$. For each compact
object $C$ in $\KInj{kE}$ and each $\fq\in\mcV_E$ with $\res^*_{G,E}(\fq)=\fp$,
Proposition~\ref{pr:perf-cogen} and \eqref{eq:kos-supp} imply
\[
\cosupp_{E}(T_{\kos C\fq}(I(\fq)))\subseteq \{\fq\}\,.
\]
It follows from Theorem~\ref{th:elab} and Proposition~\ref{pr:Quillen}
that each non-zero object $X$ in $\lam^\fp\KInj{kG}$ satisfies
\[ 
T_{\kos
  C\fq}(I(\fq))\in\Coloc(\lam^\fq(X\da_E))\subseteq\Coloc(X\da_E)\,.
\]
Together with Proposition~\ref{pr:V-ind}, one thus obtains:
\[ 
T_{\kos
  C\fq}(I(\fq))\ua^G\in\Coloc(X\da_E\ua^G)\subseteq\Coloc^{\fHom}(X)\,.
\]
It remains to observe that the collection of objects $T_{\kos
  C\fq}(I(\fq))\ua^G$ cogenerates the triangulated category
$\lam^\fp\KInj{kG}$. This is a consequence of
Proposition~\ref{pr:perf-cogen} and
Lemma~\ref{le:basechange-generate}, which can be applied, thanks to
Lemma~\ref{le:resin-change} and Proposition~\ref{pr:Quillen}.
\end{proof}

\subsection*{Applications}

The consequences of costratification described in
Section~\ref{se:costratification} apply to $\KInj{kG}$. In particular,
Theorem~\ref{th:strat-costrat} implies that $\KInj{kG}$ is stratified
as a tensor triangulated category by $H^{*}(G,k)$, which is
\cite[Theorem~9.7]{\bik:2008b}. A modification of
Theorem~\ref{th:basechange-cosupport} yields the following subgroup
theorem for cosupport, which is analogous to the one for support; see
\cite[Theorem~11.2]{\bik:2008b}.

\begin{theorem}
\label{th:cosupp-res}
Let $H$ be a subgroup of $G$. Each object $X$ in $\KInj{kG}$ satisfies
\[
\cosupp_H(X\da_H)=(\res_{G,H}^*)^{-1}(\cosupp_GX).
\]
\end{theorem}
\begin{proof}
Restriction and induction form an adjoint pair of functors satisfying the assumptions of
Theorem~\ref{th:basechange-cosupport}. Thus one gets
\[
\cosupp_H(X\da_H)\subseteq(\res_{G,H}^*)^{-1}(\cosupp_GX).
\]
For the other inclusion, one uses stratification of $\KInj{kG}$ as follows.

Fix a $\fq\in\mcV_H$ with $\fq\not\in\cosupp_H (X\da_H)$.  We need to
show that $\fp=\res_{G,H}^*(\fq)$ is not in $\cosupp_GX$.  Using the
adjunction formula for function objects from
Proposition~\ref{pr:natisos}, one has
\[\Hom_k(\gam_\fq -,X\da_H)\cong\Hom_k(-,\lam^\fq X\da_H)=0.\]
This implies $\Hom_k((\gam_\fq -)\ua^G,X)=0$, since
\[\Hom_k(U,V\da_H)\ua^G\cong\Hom_k(U\ua^G,V)\]
for all $U$ in $\KInj{kH}$ and $V$ in $\KInj{kG}$, by
\cite[Proposition~3.3]{Benson:1991a}.  There exists some  $U$ in
$\KInj{H}$ such that $(\gam_\fq U)\ua^G\ne 0$ since induction is
faithful on objects. Moreover, $(\gam_\fq U)\ua^G$ belongs to
$\gam_\fp\KInj{kG}$, by Corollary~\ref{cor:basechange-support}.  Since
$\KInj{kG}$ is stratified as a tensor triangulated category by
$H^{*}(G,k)$, the subcategory $\gam_\fp\KInj{kG}$ contains no
non-trivial tensor ideal localizing subcategories, and hence coincides
with $\Loc^\otimes((\gam_\sfq U)\ua^G)$.  Thus \[ 0=\Hom_k(\gam_{\fp}-,X)\cong
\Hom_k(-,\lam^{\fp}X),
\]
again by Proposition~\ref{pr:natisos}, and therefore
$\fp\not\in\cosupp_R X$.
\end{proof}

\subsection*{Stable module category}
Let $\Mod(kG)$ denote the (abelian) category of all $kG$-modules. The
stable module category, $\StMod(kG)$, has the same objects as the
module category $\Mod(kG)$, but the morphisms in $\StMod(kG)$ are
given by quotienting out those morphisms in $\Mod(kG)$ that factor
through a projective module. The category $\StMod(kG)$ is tensor
triangulated, with triangles induced from short exact sequences of
$kG$-modules, and product the tensor product over $k$ with diagonal
action. The trivial module $k$ is the unit of the product, and its
graded endomorphism ring in $\StMod(kG)$ is the Tate cohomology
algebra $\widehat{H}^{*}(G,k)$. There is thus an action of
$H^{*}(G,k)$ on $\StMod(kG)$, via the natural map
$H^{*}(G,k)\to\widehat{H}^{*}(G,k)$.

We regard $\StMod(kG)$ as a triangulated subcategory of $\KInj{kG}$,
as in \cite[\S6]{Benson/Krause:2008a}.

\begin{proposition}
\label{pr:stmod-kinj}
The functor $\StMod(kG)\to\KInj{kG}$ that takes a $kG$-module to its
Tate resolution identifies the tensor triangulated category $\StMod(kG)$ with the 
(co)localizing subcategory consisting of all acyclic complexes. \qed
\end{proposition}

The action of $H^{*}(G,k)$ on $\StMod(kG)$ is compatible with this identification.
Moreover the acyclic complexes form a localizing and
colocalizing subcategory of $\KInj{kG}$ that is Hom and tensor
closed. Note that a complex is acyclic if and only if its (co)support
does not contain the maximal ideal of $H^*(G,k)$. This follows from
\cite[Proposition~9.6]{\bik:2008b} for the support, and then from
Corollary~\ref{co:cosupp2} for the cosupport.

\begin{theorem}
\label{th:stmod}
The tensor triangulated category $\StMod(kG)$ is costratified by the
canonical action of $H^*(G,k)$.
\end{theorem}

\begin{proof}
Given Proposition~\ref{pr:stmod-kinj}, the desired result follows from
Theorem~\ref{th:finite-groups}, since $\lam^\fp\StMod(kG)=\lam^\fp\KInj{kG}$ for each non-maximal
$\fp$ in $\Spec H^*(G,k)$.
\end{proof}

We can now justify the results stated in the introduction.

\begin{proof}[Proof of Theorem~\ref{ith:costratification}]
From Theorems~\ref{th:stmod} and \ref{th:strat-costrat} it follows
that the tensor triangulated category $\StMod(kG)$ is stratified and
costratified by the canonical action of $H^*(G,k)$. Thus the map
sending a subset $\mcU$ of $\mcV_G$ to the subcategory of $kG$-modules
$X$ satisfying $\supp_GX\subseteq\mcU$ yields a bijection between
subsets of $\mcV_G$ and tensor ideal localizing subcategories of
$\StMod(kG)$; see \cite[Theorem~3.8]{\bik:2008b}.  Composing this
bijection with the one between localizing and colocalizing
subcategories from Corollary~\ref{co:loccoloc} gives the desired
result.
\end{proof}

\begin{proof}[Proof of Corollary~\ref{ico:locandcoloc}]
Given Theorem~\ref{th:stmod}, this follows from Corollary~\ref{co:loccoloc}.
\end{proof}

We close this discussion on the stable module category with the
following example. There is no analogue for arbitrary tensor
triangulated categories; see Example~\ref{ex:cosuppkg} and
Proposition~\ref{pr:cosuppA}.

\begin{example}
\label{ex:cosupp-stmod}
There is an equality $\cosupp_{G}M = \supp_{G}M$ for any finite
dimensional module $M$ in $\StMod(kG)$.

Indeed, denote by $(-)^{*}$ the Brown--Comenetz duality on $\StMod(kG)$
which equals $\Omega\Hom_{k}(-,k)$ by
\cite[Proposition~11.6]{\bik:2008b}, where $\Omega N$ denotes the
kernel of a projective cover of a $kG$-module $N$. Then one gets the
first equality below because $M$ is finite dimensional:
\[
\cosupp_{G}M= \cosupp_{G}M^{**}=\supp_{G}M^{*}=\supp_{G}M.
\]
The second equality is Proposition~\ref{pr:cosupp-one}, and  the last one
is well-known; see \cite[Theorem~5.1.1]{Benson:1991b}.
\end{example}

\subsection*{Modules}

Although $\Mod(kG)$ is not a triangulated category, we define
colocalizing subcategories in an analogous way. A full subcategory of
$\Mod(kG)$ is said to be \emph{thick} if whenever two modules in a
short exact sequence are in, then so is the third.  A
\emph{colocalizing subcategory} $\sfS$ of $\Mod(kG)$ is a thick
subcategory closed under all products, that is, for any family of
modules $M_i$ ($i\in I$) in $\sfS$ the product $\prod_i M_i$ is in
$\sfS$. The next lemma describes an extra tensor condition for
colocalizing subcategories; it parallels Lemma~\ref{le:Hom closed}.

\begin{lemma}
Let $\sfS$ be a colocalizing subcategory of $\Mod(kG)$. Then the
following conditions are equivalent:
\begin{enumerate}
\item If $N$ is a simple $kG$-module and $M$ is in $\sfS$ then
  $N\otimes_k M$ is in $\sfS$.
\item If $N$ is a finitely generated $kG$-module and $M$ is in $\sfS$
  then $N\otimes_k M$ is in $\sfS$.
\item If $N$ is a $kG$-module and $M$ is in $\sfS$ then $\Hom_k(N,M)$
  is in $\sfS$.\qed
\end{enumerate}
\end{lemma}

A colocalizing subcategory of $\Mod(kG)$ is said to be \emph{Hom
  closed} if the equivalent conditions of the lemma hold. The next
result is analogous to \cite[Proposition~2.1]{\bik:2008b}.

\begin{proposition}
The canonical functor from $\Mod(kG)$ to $\StMod(kG)$ induces a one to
one correspondence between non-zero Hom closed colocalizing
subcategories of $\Mod(kG)$ and Hom closed colocalizing subcategories
of $\StMod(kG)$.\qed
\end{proposition}

Thus the classification of Hom closed colocalizing subcategories of
$\Mod(kG)$ is a consequence of the costratification of $\StMod(kG)$ formulated in
Theorem~\ref{th:stmod}.

\begin{appendix}

\section{Localization functors and their adjoints}

\subsection*{Localization functors}
A functor $L\col\sfC\to\sfC$ is called \emph{localization functor} if
there exists a morphism $\eta\col\Id_\sfC\to L$ such that the morphism
$L\eta\col L \to L^2$ is invertible and $L\eta=\eta L$. The morphism
$\eta$ is called \emph{adjunction}. Recall that a morphism $\eta \col
F\to F'$ between functors is \emph{invertible} if $\eta X\col FX\to
F'X$ is invertible for each object $X$.  A functor
$\gam\col\sfC\to\sfC$ is by definition a \emph{colocalization functor}
if its opposite functor $\gam^\op\col\sfC^\op\to\sfC^\op$ is a
localization functor. The corresponding morphism $\gam\to\Id_\sfC$ is
called \emph{coadjunction}.

Any localization functor $L\col \sfC\to\sfC$ can be written as the
composite of two functors that form an adjoint pair. In order to
define these functors let us denote by $\Im L$ the \emph{essential
  image} of $L$, that is, the full subcategory of $\sfC$ that is
formed by all objects isomorphic to one of the form $LX$ for some $X$
in $\sfC$.

\begin{lemma}\label{le:loc1}
Let $L\col\sfC\to\sfC$ be a localization functor and write $L$ as a
composite \[\sfC\stackrel{F}\longrightarrow\Im
L\stackrel{G}\longrightarrow\sfC\] where $G$ denotes the inclusion
functor. Then $G$ is a right adjoint of $F$.
\end{lemma}
\begin{proof}
See \cite[Lemma~3.1]{\bik:2008a}
\end{proof}

Next we characterize the adjoint pairs of functors $(F,G)$ such that
its composite $L=GF$ is a localization functor.  Given any category
$\sfC$ and a class $\Si$ of morphisms in $\sfC$, we denote by $Q\col
\sfC\to\sfC[\Si^{-1}]$ the universal functor that inverts all
morphisms in $\Si$; see \cite[\S{I.1}]{GZ}. Thus $Q\phi$ is invertible
for all $\phi$ in $\Si$, and any functor $F\col\sfC\to\sfD$ such that
$F\phi$ is invertible for all $\phi$ in $\Si$ factors uniquely through
$Q$.

\begin{lemma}\label{le:loc2}
Let $F\col\sfC\to\sfD$ and $G\col\sfD\to\sfC$ be a pair of functors
such that $G$ is a right adjoint of $F$. Let $L=GF$ and $\eta\col
\Id_\sfC\to L$ be the adjunction morphism. Then the following are
equivalent:
\begin{enumerate}
\item The morphism $L\eta\col L \to L^2$ is invertible and $L\eta =
  \eta L$ (that is, $L$ is a localization functor).
\item The functor $F$ induces an equivalence
  $\sfC[\Si^{-1}]\xra{\sim}\sfD$, where $\Si$ denotes the class of
  morphisms $\phi$ in $\sfC$ such that $F\phi$ is invertible.
\item The functor $G$ is fully faithful.
\end{enumerate}
\end{lemma}
\begin{proof}
For (1) $\Lra$ (3) see \cite[Lemma~3.1]{\bik:2008a}; for (2) $\Lra$
(3) see \cite[Proposition~I.1.3]{GZ}.
\end{proof}

Let $F\col\sfC\to\sfD$ be an exact functor between triangulated
categories. Then the full subcategory $\Ker F$ consisting of all
objects that are annihilated by $F$ is a thick subcategory of $\sfC$.
Denote by $\Si$ the class of morphisms $\phi$ in $\sfC$ such that
$F\phi$ is invertible. Then a morphism in $\sfC$ belongs to $\Si$ if
and only if its cone belongs to $\Ker F$. Thus
$\sfC[\Si^{-1}]=\sfC/\Ker F$; see \cite{Verdier:1996a}.

\subsection*{The right adjoint of a localization functor}

We discuss the existence of a right adjoint of a localization functor.

\begin{lemma}\label{le:loc-adjoint}
Let $L\col\sfC\to\sfC$ be a localization functor with a right adjoint
$\gam$.  Then $\gam$ is a colocalization functor satisfying
$\Im\gam=\Im L$. If $\eta\col \Id_\sfC\to L$ denotes the adjunction of
$L$, then the coadjunction $\theta \col \gam\to \Id_\sfC$ of $\gam$ is
given by the composite
\[
\Hom_\sfC(X,\gam Y)\xra{\sim}\Hom_\sfC(LX,Y)\xra{(\eta
  X,Y)}\Hom_\sfC(X,Y),\quad X,Y\in\sfC\,.
\]
\end{lemma}

\begin{proof}
Fix an object $Y$ in $\sfC$. We need to show that $\gam(\theta Y)$ is
an isomorphism and that $\gam(\theta Y)=\theta (\gam Y)$.  Denote by
\[
\phi_{X,Y}\col\Hom_\sfC(X,\gam Y)\xra{\sim}\Hom_\sfC(LX,Y)
\] 
the natural bijection given by the adjunction between $L$ and $\gam$.
We use that $L(\eta X)$ is an isomorphism and that $L(\eta X)=\eta
(LX)$. The following commutative diagram shows that $\gam(\theta Y)$
is an isomorphism since all horizontal maps are bijections.
\[
\xymatrix{ \Hom_\sfC(X,\gam^2Y)\ar[r]^-{\phi_{X,\gam
      Y}}\ar[d]_-{(X,\gam\theta Y)} &\Hom_\sfC(LX,\gam
  Y)\ar[r]^-{\phi_{LX,Y}} \ar[d]^-{(LX,\theta Y)}
  &\Hom_\sfC(L^2X,Y)\ar[d]^-{(\eta LX,Y)}\\ \Hom_\sfC(X,\gam
  Y)\ar[r]^-{\phi_{X,Y}} &\Hom_\sfC(LX,Y)\ar[r]^-\id&\Hom_\sfC(LX,Y)}
\]
Combining the previous diagram with the next one shows that
$\gam(\theta Y)=\theta (\gam Y)$.
\[
\xymatrix{ \Hom_\sfC(X,\gam^2Y)\ar[r]^-{\phi_{X,\gam
      Y}}\ar[d]_-{(X,\theta \gam Y)} &\Hom_\sfC(LX,\gam
  Y)\ar[r]^-{\phi_{LX,Y}} \ar[d]^-{(\eta X,\gam Y)}
  &\Hom_\sfC(L^2X,Y)\ar[d]^-{(L\eta X,Y)}\\ \Hom_\sfC(X,\gam
  Y)\ar[r]^-\id &\Hom_\sfC(X,\gam
  Y)\ar[r]^-{\phi_{X,Y}}&\Hom_\sfC(LX,Y)}
\]
It remains to show that $\Im L=\Im\gam$. Suppose that $X$ belongs to
$\Im L$. Then $\eta X$ is invertible, and therefore the map
$\Hom_\sfC(X,\theta Y)$ is bijective for all $Y$. Using this fact for
$Y=X$ and $Y=\gam X$ shows that $\theta X$ is invertible. Thus $\Im
L\subseteq\Im\gam$. The proof of the other inclusion is similar.
\end{proof}

\begin{proposition}
\label{pr:rightadjoint}
For a localization functor $L\col\sfC\to\sfC$ are equivalent:
\begin{enumerate}
\item The functor $L$ admits a right adjoint.
\item The inclusion functor $G\col\Im L\to\sfC$ admits a right adjoint
  $G_\rho$.
\item There exists a colocalization functor $\gam\col\sfC\to\sfC$ such
  that $\Im \gam=\Im L$.
\end{enumerate}
In that case $\gam\cong GG_\rho$ and $\gam$ is a right adjoint of $L$.
\end{proposition}
\begin{proof}
(1) $\Rightarrow$ (3): See Lemma~\ref{le:loc-adjoint}.

(3) $\Rightarrow$ (2): See Lemma~\ref{le:loc1}.

(2) $\Rightarrow$ (1): Write $L$ as composite $L=GF$ as in
  Lemma~\ref{le:loc1}. Then the composite $GG_\rho$ is a right adjoint
  of $L$, since $G$ is a right adjoint of $F$.
\end{proof}

\subsection*{The left adjoint of a localization functor}

We discuss the existence of a left adjoint of a localization functor.

\begin{proposition}
\label{pr:leftadjoint}
For a localization functor $L\col\sfC\to\sfC$ are equivalent:
\begin{enumerate}
\item The functor $L$ admits a left adjoint.
\item The functor $F\col\sfC\to\Im L$ sending $X$ in $\sfC$ to $LX$
  admits a left adjoint $F_\lambda$.
\item There exists a colocalization functor $\gam\col\sfC\to\sfC$ such
  that $\gam\phi$ is invertible if and only $L\phi$ is invertible for
  each morphism $\phi$ in $\sfC$.
\end{enumerate}
In that case $\gam\cong F_\lambda F$ and $\gam$ is a left adjoint of
$L$.
\end{proposition}

\begin{proof}
Denote by $\Si$ be the class of morphisms in $\sfC$ that are inverted
by $L$.

(1) $\Rightarrow$ (2): Let $L_\lambda$ be a left adjoint of $L$ and
write $L=GF$ as in Lemma~\ref{le:loc1}. Then we have for $X,Y$ in
$\sfC$
\[
\Hom_\sfC(X,FY)\cong \Hom_\sfC(GX,GFY) \cong \Hom_\sfC(L_\lambda
GX,Y)\,.
\] 
Thus $L_\lambda G$ is a left adjoint of $F$.

(2) $\Rightarrow$ (3): It follows from Lemma~\ref{le:loc2} that $F$
induces an equivalence $\sfC[\Si^{-1}]\xra{\sim}\Im L$. Applying the
dual assertion of this lemma shows that $F_\lambda$ is fully faithful
and that $\gam=F_\lambda F$ is a colocalization functor. Moreover, the
class of morphisms that are inverted by $\gam$ coincides with the
corresponding class for $F$, which equals  $\Si$.

(3) $\Rightarrow$ (1): Denote by $Q\col\sfC\to\sfC[\Si^{-1}]$ the
universal functor inverting $\Si$ and let $\bar
L,\bar\gam\col\sfC[\Si^{-1}]\to\sfC$ be the induced functors
satisfying $L=\bar LQ$ and $\gam=\bar\gam Q$.  It follows from
Lemma~\ref{le:loc2} that $\bar L$ is a right adjoint of $Q$ and that
$\bar \gam$ is a left adjoint of $Q$. Thus for $X,Y$ in $\sfC$ we have
\[
\Hom_\sfC(\bar\gam QX,Y)\cong \Hom_\sfC(QX,QY) \cong \Hom_\sfC(X,\bar
LQY)\,.
\] 
It follows that $\gam$ is a left adjoint of $L$.
\end{proof}

\begin{remark}
\label{re:leftadjoint}
Let $\sfC$ be a triangulated category and $L\col\sfC\to\sfC$ an exact
localization functor. Then statements (2) and (3) in
Proposition~\ref{pr:leftadjoint} admits the following equivalent
reformulations:
\begin{enumerate}
\item[(2$'$)] The quotient functor $\sfC\to\sfC/\Ker L$ admits a left
  adjoint.
\item[(3$'$)] There exists an exact colocalization functor
  $\gam\col\sfC\to\sfC$ such that $\Ker \gam=\Ker L$.
\end{enumerate}
The equivalence $\sfC/\Ker L\xra{\sim}\Im L$ has already been
mentioned. The reformulation of (3) relies on the same argument: an
exact functor inverts a morphism $\phi$ in $\sfC$ if and only if it
annihilates the cone of $\phi$.
\end{remark}

\end{appendix}

\bibliographystyle{amsplain}

\end{document}